\documentclass[12pt]{article}
\title{\textbf{ \Large
Polyhedral realization of the crystal bases for extremal weight modules
over  quantized hyperbolic Kac-Moody algebras of rank $2$
}}
\author{Ryuta Hiasa\\
	{\small Graduate School of Pure and Applied Sciences, University of Tsukuba,}\\
	{\small 1-1-1 Tennodai, Tsukuba, Ibaraki 305-8571, Japan}\\
	{\small (e-mail: \texttt{hiasa@math.tsukuba.ac.jp})}
}
\date{}
\usepackage{amsmath,amssymb,amsthm}
\usepackage{bm}
\usepackage{enumerate}%\begin{enumerate}[\upshape(i)]
\usepackage{mathtools}
\mathtoolsset{showonlyrefs=true}
\usepackage[dvipdfm, top=28truemm, bottom=28truemm, left=28truemm, right=28truemm]{geometry}
\usepackage{amscd}%commutative diagram
\allowdisplaybreaks[1]
\newcommand{\Fg}{\mathfrak{g}}

\newcommand{\C}{\mathbb{C}}
\newcommand{\R}{\mathbb{R}}
\newcommand{\Q}{\mathbb{Q}}
\newcommand{\Z}{\mathbb{Z}}

\newcommand{\CB}{\mathcal{B}}
\newcommand{\wt}{\mathrm{wt}}
\newcommand{\e}{\tilde{e}}
\newcommand{\f}{\tilde{f}}
\newcommand{\ep}{\varepsilon}
\newcommand{\ph}{\varphi}
\newcommand{\al}{\alpha}
\newcommand{\alc}{\alpha^{\vee}}
\newcommand{\be}{\beta}
\newcommand{\ga}{\gamma}
\newcommand{\la}{\lambda}
\newcommand{\ze}{\zeta}
\newcommand{\0}{\mathbf{0}}

\newcommand{\CT}{\mathcal{T}}
\newcommand{\resp}[1]{(\text{resp., }#1)}
\newcommand{\pair}[2]{\langle #1, #2 \rangle}

\newcommand{\Img}{\mathrm{Im}}
\newcommand{\Zp}{\mathbb{Z}_{\geq 0, \iota^+}^{+\infty }}
\newcommand{\Zm}{\mathbb{Z}_{\leq 0, \iota^-}^{-\infty }}
\newcommand{\imp}{\Img(\Psi^+_{\iota^+})}
\newcommand{\imm}{\Img(\Psi^-_{\iota^-})}
\newcommand{\ipl}{\Img(\Psi_\iota^\lambda)}
\newcommand{\Zil}{\Z_\iota(\lambda)}
\newcommand{\Sil}{\Sigma_\iota(\lambda)}
\newcommand{\Xil}{\Xi_\iota[\lambda]}

\newcommand{\bars}{F}
\newcommand{\barb}{\bar{\beta}}

\newcommand{\vx}{\vec{x}}
\newcommand{\vy}{\vec{y}}
\newcommand{\vz}{\vec{z}}

\newcommand{\hx}{\hat{x}}
\newcommand{\hy}{\hat{y}}
\newcommand{\hz}{\hat{z}}
\newtheorem{theorem}{Theorem}[section]
\newtheorem{proposition}[theorem]{Proposition}
\newtheorem{lemma}[theorem]{Lemma}
\newtheorem{corollary}[theorem]{Corollary}
\theoremstyle{definition}
\newtheorem{definition}[theorem]{Definition}

\newtheorem{remark}[theorem]{Remark}

\numberwithin{equation}{section}
\begin{document}
\maketitle
\begin{abstract}
	Let $\mathfrak{g}$ be a hyperbolic Kac-Moody algebra of rank $2$.
	We give a polyhedral realization of
	the crystal basis for the extremal weight module of extremal weight $\lambda$,
	where $\lambda$ is an integral weight whose Weyl group orbit has
	neither a dominant integral weight nor an antidominant integral weight.
\end{abstract}

\setlength{\baselineskip}{15pt}
\section{Introduction.}
Let $A=(a_{ij})_{i,j \in I}$ be a symmetrizable generalized Cartan matrix,
where $I$ is the index set.
Let $\mathfrak{g}= \mathfrak{g}(A)$ be the Kac-Moody algebra
associated to $A$ over $\C$,
and  $U_q(\Fg)$ the quantized universal enveloping algebra over $\C(q)$
associated to $\Fg$.
We denote by $W$ the Weyl group of $\Fg$.
Let $P$ be an integral weight lattice of $\Fg$,
and $P^+$ \resp{$-P^+$} the set of dominant \resp{antidominant} integral weights in $P$.
Let $\mu \in P$ be an arbitrary integral weight.
The extremal weight module $V(\mu)$ of extremal weight $\mu$ is
the integrable $U_q(\Fg)$-module generated by a single element $v_\mu$
 with the defining relation that
 $v_\mu$ is an extremal weight vector of weight $\mu$ in the sense
of \cite[Definition 8.1.1]{K}.
This module was introduced by Kashiwara \cite{K}
as a natural generalization of integrable highest (or lowest) weight modules;
in fact,
if $\mu \in P^+$ \resp{$\mu \in -P^+$}, then
the extremal weight module of extremal weight $\mu$ is
isomorphic, as a $U_q(\Fg)$-module, to  the integrable highest \resp{lowest} weight module
of highest \resp{lowest} weight $\mu$.
Also, he proved that $V(\mu)$ has a crystal basis $\CB(\mu)$ for all $\mu \in P$.
We know from \cite[Proposition 8.2.2 (iv) and (v)]{K} that $V(\mu)\cong V(w\mu)$ as $U_q(\Fg)$-modules,
and $\CB(\mu)\cong \CB(w\mu)$ as crystals for all $\mu \in P$ and $w \in W$.
Hence we are interested in the case that $\mu$ is an integral weight such that
\begin{equation}\label{i.eq.A}
	W\mu  \cap (P^+ \cup -P^+) = \emptyset.
\end{equation}

Assume that $\Fg$ is the hyperbolic Kac-Moody algebra associated to
the generalized Cartan matrix
\begin{equation}\label{i.eq.gcm}
	A=
	\begin{pmatrix}
		2 & -a_1 \\-a_2 & 2\\
	\end{pmatrix}, \
	\text{ where }
	a_1, a_2 \in \Z_{\geq 1}
	\text{ with }
	a_1 a_2>4.
\end{equation}
Sagaki and Yu \cite{SY} proved that if  $\mu =\Lambda_1-\Lambda_2$,
where $\Lambda_1, \Lambda_2$ are the fundamental weights,
then the crystal basis $\CB(\mu)$
is isomorphic, as a crystal, to the crystal of Lakshmibai-Seshadri paths of shape $\mu$ in the case that  $a_1, a_2 \geq 2$;
note that $\mu=\Lambda_1-\Lambda_2$ does not satisfy condition \eqref{i.eq.A}
if $a_1=1$ or $a_2=1$ (see \cite[Remark 3.1.2]{Yu}).
After that, the author \cite{H,H2} classified the integral weights $\mu$ satisfying condition \eqref{i.eq.A}, and then
generalized the result
due to Sagaki and Yu mentioned above
to the case that
$\mu =k_1\Lambda_1-\Lambda_2$ with $1 \leq k_1 < a_1-1 $ or $\mu = \Lambda_1-k_2\Lambda_2 $ with $1 < k_2 \leq a_2-1 $.
% After that, the author classified the integral weights $\mu$ satisfying condition \eqref{i.eq.A}  (see  \cite[Theorem 3.1]{H}),
% and then generalized the result
% due to Sagaki and Yu mentioned above to the case that
% $\mu =k_1\Lambda_1-\Lambda_2$ with $1 \leq k_1 < a_1-1 $ or $\mu = \Lambda_1-k_2\Lambda_2 $ with $1 < k_2 \leq a_2-1 $ (see \cite[Corollary 3.4]{H2}).
The aim of this paper is to provide an explicit polyhedral realization
of $\CB(\mu)$ for arbitrary integral weight $\mu$
satisfying condition \eqref{i.eq.A}.

In this paper, we use the following
realization of the crystal basis $\CB(\mu)$
for the extremal weight module $V(\mu)$;
here, we explain it in general Kac-Moody setting.
Let $\CB(\infty)$ \resp{$\CB(-\infty)$} be
the crystal basis of the  negative \resp{positive} part of $U_q(\Fg)$.
Nakashima and Zelevinsky \cite{NZ} introduced an embedding
$\Psi_{\iota^+}^{+} :\CB(\infty)  \hookrightarrow \Zp$
of crystals,
where $\iota^+$ is an infinite sequence of elements in the index set $I$
satisfying certain condition,
and
$\Zp:=\{ (\ldots, x_k, \ldots, x_2, x_1)
\mid x_k \in \Z_{\geq 0} \text{ and } x_k=0 \text{ for } k \gg 0 \}$
is the semi-infinite $\Z$-lattice together with a crystal structure associated to $\iota^{+}$  (see \S \ref{sec.poly} below).
Assuming a certain positivity condition on $\iota^+$,
they gave a combinatorial description of $\CB(\infty)$
(which is called a polyhedral realization of $\CB(\infty)$)
as a polyhedral convex cone in $\Zp$.
Namely, they found the set $\Xi_{\iota^+}$ of linear functions on $\R^{+\infty}$ such that
the image $\Img(\Psi_{\iota^+}^{+}) \cong \CB(\infty)$ is identical to the set
\begin{equation}\label{i.eq.aim}
	\{ \hx \in \Zp  \mid \phi (\hx) \geq 0  \text{ for all } \phi \in \Xi_{\iota^+}\}.
\end{equation}
Similarly,
there exists an embedding
$\Psi_{\iota^-}^{-} : \CB(-\infty)  \hookrightarrow \Zm$ of crystals,
where $\iota^-$ is an infinite sequence of elements in the index set $I$
satisfying certain condition,
and
$\Zm :=\{ (x_0, x_{-1}, \ldots, x_k, \ldots)
\mid x_k \in \Z_{\leq 0} \text{ and } x_k=0 \text{ for } k \ll 0  \}$
is the semi-infinite $\Z$-lattice together with a crystal structure associated to $\iota^{-}$.
Hence there exists an embedding
\begin{equation}
	\Psi_{\iota}^{\mu} : \CB(\infty) \otimes \CT_\mu \otimes  \CB(-\infty) \hookrightarrow
	 \Zp \otimes \CT_\mu \otimes \Zm =: \Z_\iota(\mu)
\end{equation}
of crystals,
where $\CT_\mu $ is the crystal
consisting of a single element of weight $\mu$
(see \S \ref{sec.crystal} below),
and $\iota:=(\iota^+, \iota^-)$.
Now, in \cite{K}, Kashiwara showed that $\CB(\mu)$ is  isomorphic, as a crystal,
to the subcrystal
$\{b\in \CB(\infty) \otimes \CT_\mu \otimes \CB(-\infty)
\mid b^\ast  \text{ \rm{is extremal}} \}$
of $\CB(\infty) \otimes \CT_\mu \otimes  \CB(-\infty)$.
Therefore the crystal basis $\CB(\mu)$
is isomorphic, as a crystal, to the subcrystal
$\{ \vx \in \Img(\Psi_{\iota}^{\mu}) \mid  \vx^\ast  \text{ \rm{is extremal}} \}$
of $\Img(\Psi_{\iota}^{\mu}) =  \Img(\Psi_{\iota^+}^{+}) \otimes \CT_\mu \otimes  \Img(\Psi_{\iota^-}^{-})
\cong \CB(\infty) \otimes \CT_\mu \otimes  \CB(-\infty)$.
In this paper, we give a  polyhedral realization (such as \eqref{i.eq.aim})
of $\CB(\mu) \hookrightarrow \Img(\Psi_{\iota}^{\mu})$
in the case that $\Fg$ is of rank $2$, and $\mu$ satisfies condition \eqref{i.eq.A}.

Here we turn to be our rank $2$ case,
where $A$ is as \eqref{i.eq.gcm} with $I=\{ 1, 2\}$.
% Hoshino \cite{Ho} studied the polyhedral realization of $\CB(\infty) \otimes \CT_\mu \otimes  \CB(-\infty) \cong \Img(\Psi_{\iota}^{\mu})$
% for $\Fg=\Fg(A)$
% for some integral weights $\mu$ that do not satisfy condition \eqref{i.eq.A}.
% In this paper, as mentioned above,
% we study the polyhedral realization of $\CB(\lambda)\hookrightarrow \ipl$
% for $\Fg=\Fg(A)$
% for arbitrary integral weight $\lambda$ satisfying condition \eqref{i.eq.A}.
Let $\iota =(\iota^+, \iota^-)$ with
$\iota^+=(\ldots, i_2, i_1):=(\ldots, 2, 1, 2, 1)$
and $\iota^-=(i_{0}, i_{-1}, \ldots):=(2,1,2,1, \ldots)$.
We know that $\mu$ satisfies condition \eqref{i.eq.A}
if and only if $W\mu$ contains $\lambda$ of the form in Theorem \ref{thm.weight};
since $\CB(\mu)\cong \CB(w\mu)$, we may assume from the begining
that $\mu$ is equal to $\lambda$ of the form in Theorem \ref{thm.weight}.
For $k \in \Z$, we define the linear function $\ze_k \in (\R^\infty)^\ast$ by $\ze_k(\vx):=x_k$ for $\vx = (\ldots, x_2,x_1) \otimes t_\lambda \otimes (x_{0}, x_{-1}, \ldots) \in \R^\infty$, and set
\begin{align}
	\Xil=
	& \{\ga_0 p_0 +\ga_0 \ze_{0}- \ze_{1}, \;  \ga_1 p_1 +\ze_{0}-  \ga_1 \ze_{1}  \}\\
	&\cup\{p_{k} - \ze_{k} , \; \ga_{k}  \ze_{k}-\ze_{k+1} , \; \ga_{k+1} p_{k+1} -p_{k} +\ze_{k} - \ga_{k+1} \ze_{k+1} \mid k\geq 1 \}\\
	&\phantom{aa}\cup \{ p_{k} + \ze_{k}, \; \ze_{k-1}-\ga_{k} \ze_{k} , \; \ga_{k-1} p_{k-1} -p_{k} + \ga_{k-1} \ze_{k-1} - \ze_{k} \mid  k \leq 0   \},
\end{align}
where the numbers $\ga_k \in \R\setminus \Q$, $k \in \Z$, are defined by \eqref{eq.gamma},
and  the sequence $\{ p_m \}_{m\in \Z} $ are defined by \eqref{eq.pm} and \eqref{eq.pm2}.
We set
\begin{equation}
	\Sil:=\{ \vx \in \Zil \mid \varphi(\vx) \geq 0 \text{ for all } \varphi \in \Xi_\iota[\lambda] \}.
\end{equation}
We prove the following theorems.
\begin{theorem}[= Theorem \ref{thm.silsc}]\label{i.thm.silsc}
	The set $\Sil$ is a subcrystal of $\ipl$.
\end{theorem}
\begin{theorem}[= Corollary \ref{thm.eqality}]\label{i.thm.eqality}
	The equality $\Sil=\{ \vx \in \ipl \mid  \vx^\ast \text{ \rm is extremal}\}$ holds. Therefore, $\Sil$ is isomorphic, as a crystal, to the crystal basis $\CB(\lambda)$ of the extremal weight module $V(\lambda)$ of extremal weight $\lambda$.
\end{theorem}

This paper is organized as follows.
In \S \ref{sec.review}, we fix our notation, and recall some basic facts about extremal weight modules and their crystal bases.
In \S \ref{sec.mainresult}, we state our main results,
and in \S \ref{sec.prf}, we prove them.
In Appendix \ref{apd.1}, we give some formulas of the operators $\bars_k$ (which is defined in \S \ref{seq.prf.thm.silsc}) on $\Xil$.
In Appendix \ref{apd.2}, we give a proof of Theorem \ref{thm.weight}
in the case that $a_1=1$ or $a_2=1$.

\section{Review.}\label{sec.review}
\subsection{Kac-Moody algebras.}
Let $A=(a_{ij})_{i,j \in I}$ be a symmetrizable generalized Cartan matrix,
and  $\mathfrak{g}= \mathfrak{g}(A)$ the Kac-Moody algebra
associated to $A$ over $\C$.
We denote by $\mathfrak{h}$ the Cartan subalgebra
of $\mathfrak{g}$,
$\{ \alpha_{i} \}_{i \in I} \subset \mathfrak{h^\ast}$
the set of simple roots,
and $\{ \alpha_i^\vee \}_{i \in I} \subset \mathfrak{h}$
the set of simple coroots.
Let  $s_i$ be the simple reflection with respect to  $\alpha_i$
for $i \in I$,  and let
$W=\langle s_i \mid i \in I \rangle$
be the Weyl group of $\mathfrak{g}$.
For a positive real root $\beta$,
we denote by $\beta^{\vee}$ the dual root of $\beta$,
and by $s_\beta \in W $ the reflection with respect to $\beta$.
Let
$\{ \Lambda_i \}_{i \in I} \subset \mathfrak{h^\ast}$ be the fundamental weights for $\mathfrak{g}$,
i.e., $\langle \Lambda_i , \alpha_j^\vee \rangle=\delta_{i, j}$ for $i, j \in I$,
where $\langle \cdot , \cdot \rangle : \mathfrak{h^\ast} \times \mathfrak{h
}\rightarrow \C$
is the canonical pairing of $\mathfrak{h^\ast}$ and $\mathfrak{h}$.
We take an integral weight lattice $P$ containing $\al_i$ and $\Lambda_i$ for all  $i \in I$.
We denote by $P^+$ \resp{$-P^+$} the set of dominant
\resp{antidominant} integral weights in $P$.

Let $U_q(\Fg)$ be the quantized universal enveloping algebra over $\C(q)$
associated to $\Fg$,
and let $U_q^+(\Fg)$ \resp{$U_q^-(\Fg)$} be the positive \resp{negative}
part of $U_q(\Fg)$, that is, $\C(q)$-subalgebra generated by the Chevalley generators $E_i$ \resp{$F_i$}
of $U_q(\Fg)$ corresponding to the positive \resp{negative} simple roots $\al_i$
\resp{$-\al_i$} for $i\in I$.

\subsection{Crystal bases and crystals.}\label{sec.crystal}
For details on crystal bases and crystals,
we refer the reader to \cite{Kocb} and \cite{HK}.
Let $\CB(\infty)$ \resp{$\CB(-\infty)$} be
the crystal basis of  $U_q^-(\Fg)$
(resp., $U_q^+(\Fg)$),
and let $u_\infty \in \CB(\infty)$ (resp., $u_{-\infty} \in \CB(-\infty)$)
be the element corresponding to $1 \in U_q^-(\Fg)$ \resp{$1 \in U_q^+(\Fg)$}.
Denote by $\ast: \CB(\pm \infty) \rightarrow \CB(\pm \infty)$ the $\ast$-operation on $\CB(\pm \infty)$;
see \cite[Theorem 2.1.1]{Kdem} and \cite[\S 8.3]{Kocb}.
For $\mu \in P$, let $\CT_\mu=\{ t_\mu\}$ be the crystal consisting of a single element $t_\mu$
such that
\begin{align}
	\wt(t_\mu)=\mu,
	\quad \e_i t_\mu =\f_i t_\mu= \0,
	\quad \ep_i(t_\mu)=\ph_i(t_\mu)=-\infty
	\text{ for }  i\in I,
\end{align}
where $\0$ is an extra element not contained in any crystal.

Let $B$ be a normal crystal in the sense of \cite[\S 1.5]{K}.
We know from \cite[\S 7]{K} (see also \cite[Theorem 11.1]{Kocb}) that  $B$ has
the following action of the Weyl group $W$.
For $i\in I $ and $b \in B$, we set
\begin{align}
	S_ib:=
	\begin{cases*}
		\f_i^{\pair{\wt(b)}{\alc_i}}b & if  $\pair{\wt(b)}{\alc_i}\geq 0$,\\
		\e_i^{-\pair{\wt(b)}{\alc_i}}b &  if  $\pair{\wt(b)}{\alc_i}\leq 0$.\\
	\end{cases*}
\end{align}
Then, for $w \in W$, we set $S_w:=S_{i_1} \cdots S_{i_k}$ if $w=s_{i_1} \cdots s_{i_k}$.
Notice that $\wt(S_w b)$= $w\wt(b)$ for $w\in W$ and $b\in B$.

\begin{definition}
	An element of a normal crystal $B$ is said to be extremal if for each $w\in W$ and $i\in I$,
	\begin{align}
		\e_i(S_w b)=\0 &  \text{ if }  \pair{\wt(S_w b)}{\alc_i}\geq 0,\\
		 \f_i(S_w b)=\0 &  \text{ if }   \pair{\wt(S_w b)}{\alc_i}\leq 0.
	\end{align}
\end{definition}
Let $B$ be a normal crystal.
For $b \in B$ and  $i \in I$, we set
\begin{equation}
	\e_i^{\mathrm{max}}b:= \e_i^{\ep_i(b)}b \quad \text{and} \quad
	\f_i^{\mathrm{max}}b:= \f_i^{\ph_i(b)}b.
\end{equation}

\subsection{Crystal bases of extremal weight modules.}
Let $\mu \in P$ be an arbitrary integral weight. The extremal weight module $V(\mu)$ of extremal
weight $\mu$ is, by definition, the integrable $U_q(\Fg)$-module generated by a single element
 $v_\mu$ with the defining relation that $v_\mu$ is an extremal weight vector of weight $\mu$ in the sense
of \cite[Definition 8.1.1]{K}. We know from \cite[Proposition 8.2.2]{K} that $V(\mu)$  has a crystal basis $\CB(\mu)$.
Let $u_\mu$ denote the element of $\CB(\mu)$ corresponding to $v_\mu$.

\begin{remark}\label{rem.extmd}
We see from \cite[Proposition 8.2.2 (iv) and (v)]{K} that
$V(\mu)\cong V(w\mu)$ as $U_q(\Fg)$-modules,
and $\CB(\mu)\cong \CB(w\mu)$ as crystals for all $\mu \in P$ and $w\in W$.
Also, we know from the comment at the end of \cite[\S8.2]{K} that
if $\mu \in P^+$ \resp{$\mu \in -P^+$},
then $V(\mu)$ is isomorphic, as a $U_q(\Fg)$-module, to the integrable highest
\resp{lowest}
weight
module of highest \resp{lowest}
weight $\mu$, and $\CB(\mu)$ is isomorphic, as a crystal, to its
crystal basis.
So, we focus on those $\mu \in P$ satisfying the condition that
\begin{align}
	W\mu \cap (P^+ \cup -P^+)=\emptyset.\label{eq.A}
\end{align}
\end{remark}

The crystal basis $\CB(\mu)$ of $V (\mu)$ can be realized (as a crystal) as follows. We set
\begin{align}
	\CB :=\bigsqcup_{\mu\in P} \CB(\infty) \otimes \CT_\mu \otimes \CB(-\infty);
\end{align}
in fact, $\CB$ is isomorphic, as a crystal, to the crystal basis $\CB(\tilde{U}_q(\Fg))$ of the modified quantized universal enveloping algebra $\tilde{U}_q(\Fg)$ associated to $\Fg$
(see \cite[Theorem 3.1.1]{K}). Denote by
$\ast : \CB \rightarrow \CB$ the $\ast$-operation on $\CB$
(see \cite[Theorem 4.3.2]{K});
we know from \cite[Corollary 4.3.3]{K} that for
$b_1 \in \CB(\infty)$, $b_2 \in \CB(-\infty)$, and $\mu \in P$,
\begin{align}\label{eq.starope}
	(b_1 \otimes t_\mu \otimes b_2)^\ast = b_1^\ast \otimes t_{-\mu -\wt(b_1)- \wt(b_2)} \otimes b_2^\ast.
\end{align}

\begin{remark}
	The weight of $(b_1 \otimes t_\mu \otimes b_2)^\ast$ is equal to $-\mu$ for all $b_1 \in \CB(\infty)$ and $b_2 \in \CB(-\infty)$ since $\wt(b_1^\ast) = \wt(b_1)$ and $\wt(b_2^\ast) = \wt(b_2)$.
\end{remark}

Because $\CB$ is a normal crystal by \cite[\S 2.1 and Theorem 3.1.1]{K}, $\CB$ has the action of
the Weyl group $W$ (see \S \ref{sec.crystal}). We know the following proposition from
\cite[Proposition 8.2.2 (and Theorem 3.1.1)]{K}.

\begin{theorem}\label{thm.cbext}%crystal basis extremal
	For $\mu \in P$, the set
	$\{b\in \CB(\infty) \otimes \CT_\mu \otimes \CB(-\infty)
	\mid b^\ast \text{\rm{ is extremal}} \}$
	is a subcrystal of $ \CB(\infty) \otimes \CT_\mu \otimes \CB(-\infty)$, and is isomorphic,
	as a crystal, to the crystal basis $\CB(\mu)$ of the extremal weight module $V(\mu)$ of extremal weight $\mu$.
In particular, $u_{\infty} \otimes t_\mu \otimes u_{-\infty} \in \CB(\infty) \otimes \CT_\mu \otimes \CB(-\infty)$ is contained in the set above, and corresponds to $u_\mu \in \CB(\mu)$ under this isomorphism.
\end{theorem}

\subsection{Realizations of $\mathcal{B}(\pm\infty)$ and $\mathcal{B}(\mu)$.}\label{sec.poly}
Let us recall realizations of $\mathcal{B}(\infty)$
and $\mathcal{B}(-\infty)$ from \cite{NZ}.
We fix an infinite sequence
$\iota^+=(\ldots, i_k, \ldots, i_2, i_1)$
of elements of  $I$ such that $i_k \neq i_{k+1} $ for $k \in \Z_{\geq 1}$, and
$	\# \{ k \in \Z_{\geq 1} \mid i_k=i \}=\infty$ for each $i \in I$.
Similarly, we fix an infinite sequence
$\iota^-=(i_0, i_{-1}, \ldots, i_k, \ldots)$
of elements of  $I$ such that $i_k \neq i_{k-1} $ for $k \in \Z_{\leq 0}$, and
$	\# \{ k \in \Z_{\leq 0} \mid i_k=i \}=\infty $ for each $i \in I$.
We set
\begin{align}
	\Z^{+\infty}_{\geq 0}
	&:=\{ (\ldots, x_k, \ldots, x_2, x_1)
	\mid x_k \in \Z_{\geq 0} \text{ and } x_k=0 \text{ for } k \gg 0 \},\\
	\Z^{-\infty}_{\leq 0}
	&:=\{ (x_0, x_{-1}, \ldots, x_k, \ldots)
	\mid x_k \in \Z_{\leq 0} \text{ and } x_k=0 \text{ for } k \ll 0  \}.
\end{align}
We endow $\Z^{+\infty}_{\geq 0}$ and $\Z^{-\infty}_{\leq 0}$ with crystal structures as follows.
Let $\hx^+=(\ldots, x_k, \ldots, x_2, x_1) \in \Z^{+\infty}_{\geq 0}$
and $\hx^-=(x_0, x_{-1} \ldots, x_k \ldots) \in \Z^{-\infty}_{\leq 0}$.
For $k \geq 1$, we set
\begin{align}
	\sigma^+_k(\hx^+) = x_k + \sum_{j>k}\pair{\al_{i_j}}{\alc_{i_k}}x_j,
\end{align}
and for $k\leq 0$, we set
\begin{align}
		\sigma^-_k(\hx^-) =-x_k - \sum_{j<k}\pair{\al_{i_j}}{\alc_{i_k}}x_j;
\end{align}
since $x_j=0$ for $|j|\gg 0$, we see that  $\sigma^\pm_k(\hx^\pm)$  is well-defined, and $\sigma^\pm_k(\hx^\pm) =0$ for $|k| \gg 0$.
For $i \in I$,  we set
$\sigma^+_{(i)}({\hx^+}):= \max \{ \sigma^+_k(\hx^+)\mid k \geq 1,   i_k =i \}$
and $\sigma^-_{(i)}({\hx^-}):= \max \{ \sigma^-_k(\hx^-)\mid k \leq 0,   i_k =i \}$,
and define
\begin{align}
	&M^+_{(i)}= M^+_{(i)}(\hx^+):=
	\{ k \mid k \geq 1, i_k=i, \sigma^+_k(\hx^+) =\sigma^+_{(i)}(\hx^+) \},\\
	&M^-_{(i)}= M^-_{(i)}(\hx^-):=
	\{ k \mid k \leq 0, i_k=i, \sigma^-_k(\hx^-) =\sigma^-_{(i)}(\hx^-) \}.
\end{align}
Note that $\sigma^\pm_{(i)}(\hx^\pm) \geq 0$, and that $M^\pm_{(i)}= M^\pm_{(i)}(\hx^\pm)$ is a finite set
if and only if $\sigma^\pm_{(i)}(\hx^\pm) > 0$.
We define the maps
$\e_i, \f_i: \Z_{\geq 0}^{+\infty} \rightarrow \Z_{\geq 0}^{+\infty} \sqcup \{ \0 \}$
and  $\e_i, \f_i: \Z_{\leq 0}^{-\infty} \rightarrow \Z_{\leq 0}^{-\infty} \sqcup \{ \0 \}$
by
\begin{align}
	\e_i \hx^+ &:=
	\begin{cases}
		(\ldots, x'_k, \ldots, x'_2, x'_1)
		\text{ with } x'_k:=x_k - \delta_{k, \max{M^+_{(i)}}}
		& \text{ if } \sigma^+_{(i)}({\hx^+}) > 0,\\
		\0 & \text{ if } \sigma^+_{(i)}({\hx^+}) = 0,
	\end{cases}\\
	\f_i \hx^+ &:=(\ldots, x'_k, \ldots, x'_2, x'_1)
	\text{ with } x'_k:=x_k + \delta_{k, \min{M^+_{(i)}}},\\
	\e_i \hx^- &:= (x'_0, x'_{-1}, \ldots, x'_k, \ldots)
	\text{ with } x'_k:=x_k - \delta_{k, \max{M^-_{(i)}}},\\
		\f_i \hx^- &:=
			 \begin{cases}
					(x'_0, x'_{-1}, \ldots, x'_k, \ldots)
					\text{ with } x'_k:=x_k + \delta_{k, \min{M^-_{(i)}}} & \text{ if } \sigma^-_{(i)}({\hx^-}) > 0,\\
				 \0 & \text{ if } \sigma^-_{(i)}({\hx^-}) = 0,
			 \end{cases}
\end{align}
respectively.
Moreover, we define
\begin{gather}
	\wt(\hx^+):=-\sum_{j\geq 1} x_j \al_{i_j}, \quad
	\ep_i(\hx^+):=\sigma^+_{(i)}(\hx^+), \quad
		 \ph_i(\hx^+):=\ep_i(\hx^+)+\pair{\wt(\hx^+)}{\alc_i},\\
	\wt(\hx^-):=-\sum_{j\leq 0} x_j \al_{i_j}, \quad
	\ph_i(\hx^-):=\sigma^-_{(i)}(\hx^-), \quad
	\ep_i(\hx^-):=\ph_i(\hx^-)-\pair{\wt(\hx^-)}{\alc_i}.
\end{gather}
These maps make $\Z_{\geq 0}^{+\infty}$ \resp{$\Z_{\leq 0}^{-\infty}$} into a crystal for $\Fg$;
we denote this crystal by $\Zp$ \resp{$\Zm$}.
\begin{theorem}[{\cite[Theorem 2.5]{NZ}}]\label{thm.polyh}
There exists an embedding
$\Psi^+_{\iota^+}: \CB(\infty) \hookrightarrow \Zp$
of crystals which sends $u_\infty \in \CB(\infty )$
to $z_{\infty}:=(\ldots, 0 \ldots, 0, 0) \in \Zp$.
Similarly, there exists an embedding
$\Psi^-_{\iota^-}: \CB(-\infty) \hookrightarrow \Zm$
of crystals which sends $u_{-\infty} \in \CB(-\infty) $
to $z_{-\infty}:=(0, 0, \ldots, 0, \ldots) \in \Zm$.
\end{theorem}
We define the $\ast$-operations on $\Img(\Psi_{\iota^\pm}^{\pm})$
by the following commutative diagram:
\begin{align}
	\begin{CD}
    \CB(\pm\infty) @>{\ast}>> \CB(\pm\infty) \\
  	@V{\Psi_{\iota^\pm}^{\pm}}VV    @VV{\Psi_{\iota^\pm}^{\pm}}V \\
    \Img(\Psi_{\iota^\pm}^{\pm})  @>{\ast}>>  \Img(\Psi_{\iota^\pm}^{\pm}).
  \end{CD}
\end{align}
We know the following proposition from \cite[Remark in \S 2.4]{NZ}.
\begin{proposition}\label{prop.element}
	Keep the notation and setting above;
	recall that
	$\iota^+=( \ldots, i_2, i_1)$ and
	$\iota^-=(i_0, i_{-1}, \ldots)$.
	\begin{enumerate}[{\upshape (1)}]
		\item  Let $\hx=(\ldots,x_2, x_1) \in \Zp$.
		Then, $\hx \in \Img(\Psi^+_{\iota^+})$ if and only if
		\begin{equation}
			0=\ep_{i_k}(\f_{i_{k+1}}^{x_{k+1}} \f_{i_{k+2}}^{x_{k+2}} \cdots z_{\infty} )
		\end{equation}
		for all $ k \geq 1$.
		Furthermore, if $\hx \in \Img(\Psi^+_{\iota^+})$, then  $\hx^\ast=\f_{i_1}^{x_1} \f_{i_2}^{x_2}\cdots z_\infty$,
		and
		$x_k=\ep_{i_k}(\e_{i_{k-1}}^{x_{k-1}}\cdots \e_{i_2}^{x_{2}} \e_{i_1}^{x_{1}} \hx^\ast )$ for  $k \geq 1$.
		\item  Let $\hx=(x_0, x_{-1},\ldots)  \in \Zm$.
		Then, $\hx \in \Img(\Psi^-_{\iota^-})$ if and only if
		\begin{equation}
			0=\ph_{i_k}(\e_{i_{k-1}}^{-x_{k-1}} \e_{i_{k-2}}^{-x_{k-2}} \cdots z_{-\infty} )
		\end{equation}
		for all $ k \leq 0$.
		Furthermore, if $\hx \in \Img(\Psi^-_{\iota^-})$, then  $\hx^\ast=\e_{i_0}^{-x_{0}} \e_{i_{-1}}^{-x_{-1}}\cdots z_{-\infty}$,
		and
		$-x_k=\ph_{i_k}(\f_{i_{k+1}}^{-x_{k+1}}\cdots \f_{i_{-1}}^{-x_{-1}} \f_{i_0}^{-x_{0}} \hx^\ast )$ for $k \leq 0$.
	\end{enumerate}
\end{proposition}

We set $ \iota:=(\iota^+, \iota^-)$, and $\Z_\iota(\mu):= \Zp \otimes \CT_\mu \otimes \Zm $ for  $\mu \in P$.
By the tensor product rule of crystals, we can describe the crystal structure of
$\Z_\iota(\mu)$ as follows.
Let $\vx=\hx^+ \otimes t_\mu \otimes \hx^- \in \Z_\iota(\mu)$
with  $\hx^+=(\ldots, x_2, x_1) \in \Zp $
and $\hx^-= ( x_0, x_{-1}, \ldots ) \in \Zm$.
For $k\in \Z $, we set
\begin{equation}
	\sigma_k(\vx):=
	\begin{cases*}
		\sigma_k^+(\hx^+) & if $k\geq 1$,\\
		\sigma_k^-(\hx^-) -\pair{\wt(\vx)}{\alc_{i_k}} &  if  $k\leq 0$.\\
	\end{cases*}
\end{equation}
For $i \in I$, we set
$\sigma_{(i)}({\vx}):= \max\{\sigma_k(\vx)\mid k\in\Z,  i_k =i\}$, and
\begin{equation}\label{eq.mi}
	M_{(i)}= M_{(i)}(\vx):=
	\{ k \mid i_k=i, \sigma_k(\vx) =\sigma_{(i)}(\vx) \}.
\end{equation}
Then we see that
\begin{equation}
	\wt(\vx) =\mu-\sum_{j\in \Z} x_j \al_{i_j}; \quad \ep_i(\vx) = \sigma_{(i)}(\vx);
	\quad \ph_i(\vx) =\ep_i(\vx)+ \pair{\wt(\vx)}{\alc_i};
\end{equation}
if $\ep_i(\vx) > 0$, then
\begin{equation}
	\e_i \vx =
		(\ldots, x'_2, x'_1)\otimes  t_\mu \otimes ( x'_0, x'_{-1}, \ldots )
		\text{ with } x'_k:=x_k - \delta_{k, \max{M_{(i)}}};
\end{equation}
if $ \ph_i(\vx) > 0$, then
\begin{equation}\label{eq.deffi}
	\f_i \vx =
		(\ldots, x'_2, x'_1)\otimes  t_\mu \otimes ( x'_0, x'_{-1}, \ldots )
		\text{ with }  x'_k:=x_k + \delta_{k, \min{M_{(i)}}};
\end{equation}
if $\ep_i(\vx) =0$, then  $\e_i \vx =\0$;
if $\ph_i(\vx) =0$, then $\f_i \vx =\0$.
% \begin{align}
% 	&\e_i \vx =
% 	\begin{cases}
% 		(\ldots, x'_2, x'_1)\otimes  t_\mu \otimes ( x'_0, x'_{-1}, \ldots )
% 		\text{ with } x'_k:=x_k - \delta_{k, \max{M_{(i)}}} & \text{ if } \ep_i(\vx) > 0,\\
% 		\0 & \text{ if } \ep_i(\vx) =0,
% 	\end{cases}\\
% 	&\f_i \vx =
% 	\begin{cases}\label{eq.deffi}
% 		(\ldots, x'_2, x'_1)\otimes  t_\mu \otimes ( x'_0, x'_{-1}, \ldots )
% 		\text{ with } x'_k:=x_k + \delta_{k, \min{M_{(i)}}} & \text{ if } \ph_i(\vx) > 0,\\
% 		\0 & \text{ if } \ph_i(\vx) =0.
% 	\end{cases}
% \end{align}
%
The next corollary follows immediately from Theorem \ref{thm.polyh}.
\begin{corollary}\label{cor.polyhofex}
	For each $\mu \in P$,
	there exists an embedding
	$\Psi_{\iota}^{\mu} := \Psi_{\iota^+}^{+} \otimes  \mathrm{id}   \otimes \Psi_{\iota^-}^{-} : \CB(\infty) \otimes \CT_\mu \otimes  \CB(-\infty) \hookrightarrow
	\Z_\iota(\mu)$
	 of crystals
	which sends
	$u_\infty  \otimes t_\mu \otimes u_{-\infty} \in \CB(\infty) \otimes \CT_\mu \otimes  \CB(-\infty)$  to
	$z_\mu:= z_{\infty} \otimes t_\mu \otimes  z_{-\infty} \in \Z_\iota(\mu)$.
\end{corollary}
We also define the $\ast$-operation on
$\Img(\Psi_{\iota}^{\mu})=\Img(\Psi_{\iota^+}^{+}) \otimes \CT_\mu \otimes \Img(\Psi_{\iota^-}^{-})$
by the following commutative diagram:
\begin{align}
	\begin{CD}
    \CB(\infty) \otimes \CT_\mu \otimes \CB(-\infty)  @>{\ast}>> \CB(\infty) \otimes \CT_\mu \otimes \CB(-\infty) \\
  	@V{\Psi_{\iota}^{\mu} }VV
		@VV{\Psi_{\iota}^{\mu}}V \\
    \Img(\Psi_{\iota^+}^{+}) \otimes \CT_\mu \otimes\Img(\Psi_{\iota^-}^{-})   @>{\ast}>>  \Img(\Psi_{\iota^+}^{+}) \otimes \CT_\mu \otimes\Img(\Psi_{\iota^-}^{-}).
  \end{CD}
\end{align}
We see  by \eqref{eq.starope} that
if $z_1 \in \Img(\Psi_{\iota^+}^{+})$ and $z_2 \in \Img(\Psi_{\iota^-}^{-})$,
then
\begin{align}\label{eq.zast}
	(z_1 \otimes t_\mu \otimes z_2)^\ast = z_1^\ast \otimes t_{-\mu -\wt(z_1)- \wt(z_2)} \otimes z_2^\ast.
\end{align}
The next corollary is a consequence of Theorem \ref{thm.cbext} and Corollary \ref{cor.polyhofex}.
\begin{corollary}\label{thm.zext}
	For $\mu \in P$, the set
	$\{\vx \in \Img(\Psi_{\iota}^{\mu})
	\mid \vx^\ast \text{ \rm is extremal}\}$
	is a subcrystal of $ \Img(\Psi_{\iota}^{\mu})$, and is isomorphic,
	as a crystal, to the crystal basis $\CB(\mu)$ of the extremal weight module $V(\mu)$ of extremal weight $\mu$.
% In particular, $u_{\infty} \otimes t_\mu \otimes u_{-\infty} \in \CB(\infty) \otimes \CT_\mu \otimes \CB(-\infty)$ is contained in the set above, and corresponds to $u_\mu \in \CB(\mu)$ under the isomorphism.
\end{corollary}

\section{Main results.}\label{sec.mainresult}
In the following, we assume that
the generalized Cartan matrix $A$ is
\begin{equation}\label{eq.gcm}
	A=
	\begin{pmatrix}
		2 & -a_1 \\-a_2 & 2\\
	\end{pmatrix}, \
	\text{ where }
	a_1, a_2 \in \Z_{\geq 1}
	\text{ with }
	a_1 a_2>4,
\end{equation}
and $\iota =(\iota^+, \iota^-)$ with
% note that $\al_1=2\Lambda_1-a_2\Lambda_2$ and  $\al_2=-a_1\Lambda_1+2\Lambda_2$.
$\iota^+=(\ldots, i_2, i_1):=(\ldots, 2, 1, 2, 1)$ and $\iota^-=(i_{0}, i_{-1}, \ldots):=(2,1,2,1, \ldots)$.
We set
\begin{equation}
	\al:=\frac{ a_1a_2 + \sqrt{a_1^2 a_2^2 -4a_1a_2}  }{2a_2}, \quad
	\be:=\frac{ a_1a_2 + \sqrt{a_1^2 a_2^2 -4a_1a_2}  }{2a_1},
\end{equation}
and
\begin{equation}\label{eq.gamma}
	\ga_k:=
	\begin{cases*}
		\alpha & if $k$ is even,\\
		\beta  & if $k$ is odd
	\end{cases*}
\end{equation}
for $k \in \Z$;
note that $\al, \be \in \R\setminus \Q$ and $\al, \be >0$.
By the definition, we have
\begin{equation}\label{eq.gamaik}
	\frac{1}{\ga_k}+\ga_{k+1}=a_{i_k}.
\end{equation}
Let $\Lambda_1, \Lambda_2$ denote the fundamental weights for $\Fg=\Fg(A)$;
note that
$P=\Z \Lambda_1 \oplus \Z \Lambda _2$.
\begin{theorem}\label{thm.weight}
	Let $\mathbb{O}:=\{ W\mu \mid \mu \in P \}$ be
	the set of $W$-orbits in $P$.
	\begin{itemize}
		\item [\upshape(1)]
			Assume that $a_1, a_2 \geq 2$. Then,
			$O\in \mathbb{O}$ satisfies condition \eqref{eq.A},
			that is,
			$O  \cap (P^+ \cup -P^+) = \emptyset$  if and only if
			$O$ contains
			an integral weight $\lambda$
			of the form either \eqref{enu.1} or \eqref{enu.2}:
			\begin{enumerate}[\upshape(i)]
				\item $\lambda=k_1\Lambda_1-k_2\Lambda_2$ for some
				$k_1, k_2 \in \Z_{>0}$ such that $k_2\leq k_1<(a_1-1)k_2$;\label{enu.1}
				\item $\lambda=k_1\Lambda_1-k_2\Lambda_2$ for some
				$k_1, k_2 \in \Z_{>0}$ such that $k_1<k_2\leq (a_2-1)k_1 $.\label{enu.2}
			\end{enumerate}
		\item [\upshape(2)]
			Assume that $a_1=1$. Then,
			$O\in \mathbb{O}$ satisfies condition \eqref{eq.A}
			if and only if
			$O$ contains
			an integral weight $\lambda$
			of the form
			$\lambda=k_1\Lambda_1-k_2\Lambda_2$  for some
			$k_1, k_2 \in \Z_{>0}$  such that  $2k_1\leq k_2 \leq  (a_2-2)k_1$.
			% \begin{equation}
			% 	\lambda=k_1\Lambda_1-k_2\Lambda_2 \text{ for some }
			% 	k_1, k_2 \in \Z_{>0} \text{ such that } 2k_1\leq k_2 \leq  (a_2-2)k_1.
			% \end{equation}
		\item [\upshape(3)]
			Assume that $a_2=1$. Then,
			$O\in \mathbb{O}$ satisfies condition \eqref{eq.A}
			if and only if
			$O$ contains
			an integral weight $\lambda$
			of the form
			$\lambda=k_1\Lambda_1-k_2\Lambda_2$  for some
			$k_1, k_2 \in \Z_{>0}$  such that  $2k_2 \leq  k_1\leq (a_1-2)k_2$.
			% \begin{equation}
			% 	\lambda=k_1\Lambda_1-k_2\Lambda_2 \text{ for some }
			% 	k_1, k_2 \in \Z_{>0} \text{ such that } 2k_2 \leq  k_1\leq (a_1-2)k_2.
			% \end{equation}
	\end{itemize}
\end{theorem}
\begin{proof}
	The assertion of part (1)
	is nothing but \cite[Theorem 3.1]{H}.
	We can show parts (2) and (3)
	in exactly the same way as \cite[Theorem 3.1]{H}; see Appendix \ref{apd.2}.
\end{proof}

Let $\lambda=k_1\Lambda_1-k_2\Lambda_2 \in P$ be an integral weight of the form mentioned  in Theorem \ref{thm.weight} above.
We define the sequence $\{ p_m \}_{m\in \Z} $ of integers
by the following recursive formulas: for $m\geq 0$,
\begin{equation}\label{eq.pm}
	p_0:=k_2, \quad
	p_1:=k_1, \quad
	p_{m+2}:=
	\begin{cases*}
		a_2 p_{m+1}-p_m & if $m$ is even, \\
		a_1 p_{m+1}-p_m & if $m$ is odd; \\
	\end{cases*}
\end{equation}
for  $m<0$,
	\begin{equation}\label{eq.pm2}
		p_{m}=
		\begin{cases*}
			a_2 p_{m+1}-p_{m+2} & if $m$ is even, \\
			a_1 p_{m+1}-p_{m+2} & if $m$ is odd; \\
		\end{cases*}
\end{equation}
it follows from \cite[Remark 3.7]{H}
(and  the comment in \cite[\S 3.1]{Yu})
that $	p_m>0 $ for all $ m \in \Z$.
We regard  $\R^{\infty} := \{ \vx = (\ldots, x_2,x_1) \otimes t_\lambda \otimes (x_{0}, x_{-1}, \ldots) \mid x_k \in \R  \text{ and }    x_k=0 \text{ for }  |k| \gg 0 \}$
as an infinite dimensional vector space over $\R$;
note that $\Zil \subset \R^{\infty}$.
Let $(\R^\infty)^\ast:= \mathrm{Hom}_\R(\R^\infty, \R)$ be its dual space.
For $k \in \Z$, we define the linear function $\ze_k \in (\R^\infty)^\ast$ by $\ze_k(\vx):=x_k$ for $\vx = (\ldots, x_2,x_1) \otimes t_\lambda \otimes (x_{0}, x_{-1}, \ldots) \in \R^\infty$.
Set
\begin{equation}
	\Sil:=\{ \vx \in \Zil \mid \varphi(\vx) \geq 0 \text{ for all } \varphi \in \Xi_\iota[\lambda] \},
\end{equation}
where
\begin{align}
	\Xil=
	& \{\ga_0 p_0 +\ga_0 \ze_{0}- \ze_{1}, \;  \ga_1 p_1 +\ze_{0}-  \ga_1 \ze_{1}  \}\\
	&\cup\{p_{k} - \ze_{k} , \; \ga_{k}  \ze_{k}-\ze_{k+1} , \; \ga_{k+1} p_{k+1} -p_{k} +\ze_{k} - \ga_{k+1} \ze_{k+1} \mid k\geq 1 \}\\
	&\phantom{aa}\cup \{ p_{k} + \ze_{k}, \; \ze_{k-1}-\ga_{k} \ze_{k} , \; \ga_{k-1} p_{k-1} -p_{k} + \ga_{k-1} \ze_{k-1} - \ze_{k} \mid  k \leq 0   \}.
\end{align}

\begin{theorem}[will be proved in \S \ref{seq.prf.thm.silsc}]\label{thm.silsc}
	The set $\Sil$ is a subcrystal of $\ipl$.
\end{theorem}

Let $\Sil'$ be the subset of $\Sil$ consisting of the elements  of the form
$\hx \otimes t_\lambda \otimes z_{-\infty}$ with $\hx \in \imp$.
\begin{theorem}[will be proved in \S \ref{seq.prf.prop.silex}]\label{prop.silex}
	For $\vx \in \Sil'$, the element $\vx^\ast$ is extremal.
\end{theorem}

\begin{theorem}[will be proved in \S \ref{seq.prf.prop.silcon}]\label{prop.silcon}
	Let $ \vy \in \ipl $. If $ \vy^\ast$ is extremal, then there exist $i_1, \ldots, i_l \in I $ and $\vx \in \Sil'$ such that $\vx= \f_{i_l} \cdots \f_{i_1}\vy$.
\end{theorem}

\begin{corollary}\label{thm.eqality}
	The equality
	$\Sil= \{\vx \in \Img(\Psi_{\iota}^{\lambda}) \mid \vx^\ast
	 \text{ \rm is extremal} \}$
	holds. Therefore, $\Sil$ is isomorphic, as a crystal, to the crystal basis $\CB(\lambda)$ of the extremal weight module $V(\lambda)$ of extremal weight $\lambda$.
\end{corollary}

\begin{proof}
	Set
	$B:=\{\vx \in \Img(\Psi_{\iota}^{\lambda})
	\mid \vx^\ast  \text{ \rm is extremal} \}$.
	First, we show that $\Sil \subset B$.
	Let $\vx = \hx_1 \otimes t_\la \otimes \hx_2 \in \Sil$ with $\hx_1 \in \Zp$ and
	$\hx_2 \in \Zm$.
	By Theorem \ref{thm.silsc}, we have $\hx_2 \in \imm$.
	Since $\imm \cong \CB(-\infty)$ as crystals,
	there exist $i_1, \ldots, i_l$ such that
	$\f_{i_l}^{\mathrm{max}} \cdots \f_{i_1}^{\mathrm{max}} \hx_2 = z_{-\infty}$.
	Then we see by the tensor product rule of crystals that
	$\vy:= \f_{i_l}^{\mathrm{max}} \cdots \f_{i_1}^{\mathrm{max}}  \vx \in \Sil'$.
	Since $\vx \in \Sil \subset \ipl$, we see that $\vy \in \ipl$.
	Also, it follows from Theorem \ref{prop.silex} that $ \vy^\ast$ is extremal.
	Thus we obtain $ \vy \in B$.
	Since $B$ is a subcrystal by Corollary \ref{thm.zext},
	we obtain $\vx \in B$.

	Next, we show that $\Sil \supset B$.
	Let $\vy \in \ipl$ be such that $\vy^\ast$ is extremal.
	We see from Theorem \ref{prop.silcon} that
	there exist $i_1, \ldots, i_l \in I $ such that $\f_{i_l} \cdots \f_{i_1}\vy \in \Sil' \subset \Sil$.
	Therefore, by Theorem \ref{thm.silsc},
	we obtain $\vy \in \Sil$.

	Thus we have proved the corollary.
\end{proof}

\section{Proofs.}\label{sec.prf}
Throughout this section, we take and fix $\iota=(\iota^{+}, \iota^{-})$ and $\lambda=k_1\Lambda_1-k_2\Lambda_2 \in P$ as in \S \ref{sec.mainresult}.

\subsection{Polyhedral realization of $\CB(\pm\infty)$ in the rank 2 case.}
We define the sequences
$\{ c_j \}_{j \geq 0} $ and $\{ c_j' \}_{j \geq 0} $
of integers by the following recursive formulas:
for $j\geq 0$,
\begin{equation}
	c_0:=0, \quad
	c_1:=1, \quad
	c_{j+2}:=
	\begin{cases*}
		a_1 c_{j+1}-c_j & if $j$ is even, \\
		a_2 c_{j+1}-c_j & if $j$ is odd; \\
	\end{cases*}
\end{equation}
\begin{equation}
	c'_0:=0, \quad
	c'_{1}:=1, \quad
	c'_{j+2}:=
	\begin{cases*}
		a_2  c'_{j+1}-c'_j & if $j$ is even, \\
		a_1  c'_{j+1}-c'_j & if $j$ is odd;\\
	\end{cases*}
\end{equation}
it is easy to show that
$c_{j}>0$ and $c'_{j}>0$ for all $j\geq 1$
(see also the comment before \cite[Theorem 4.1]{NZ}).
By \cite[Corollary 4.7]{Ho} and the fact that
$1/\beta=(a_1a_2-\sqrt{a_1^2a_2^2-4a_1a_2})/2a_2$,
we obtain the following lemma.
\begin{lemma}\label{lem.cab}
	The following sequences are strictly decreasing, and converge  to $\al$ and $\be$, respectively:
	\begin{equation}
		\frac{c_2}{c_1}> \frac{c'_3}{c'_2}> \frac{c_4}{c_3}> \frac{c'_5}{c'_4}>  \cdots \to \alpha, \qquad
		\frac{c'_2}{c'_1}> \frac{c_3}{c_2}> \frac{c'_4}{c'_3}> \frac{c_5}{c_4}>  \cdots \to \beta.
	\end{equation}
\end{lemma}
Applying \cite[Theorem 4.1]{NZ} to our rank $2$ case,
we obtain the following explicit descriptions of
the images of the maps
$\Psi^+_{\iota^+}: \CB(\infty) \to \Zp$
and $\Psi^-_{\iota^-}: \CB(-\infty)\to \Zm$.
\begin{proposition}\label{prop.impsi}
	It hold that
	\begin{align}
		\Img(\Psi^+_{\iota^+})&=
		\{(\ldots, x_2, x_1 )\in \Z^{+\infty}_{\geq 0} \mid c_j x_j - c_{j-1} x_{j+1} \geq 0 \text{\rm{ for }} j \geq 1 \},\\
		\Img(\Psi^-_{\iota^-})&=
		\{(x_0, x_{-1}, \ldots )\in \Z^{-\infty}_{\leq 0} \mid c'_{-j+1} x_j - c'_{-j} x_{j-1} \leq 0 \text{\rm{ for }} j \leq 0 \}.
	\end{align}
\end{proposition}

Recall that the sequence $\{ p_m\}_{m\in\Z}$ is defined by recursive formulas \eqref{eq.pm}
and \eqref{eq.pm2}.
Let $\hx = ( \ldots, x_2, x_1) \in \imp$ be such that  $x_m \leq p_m$ for all $m \in \Z_{\geq 1}$.
For  $l \geq 1$, we set
\begin{align}
	z_1(\hx, l) &:= (\ldots, x_2, x_1,p_0, p_{-1}, \ldots, p_{-2l+2}, p_{-2l+1}) \in \Zp,\\
	z_2(\hx, l) &:= (x_{2l} -p_{2l} , \ldots, x_2-p_2,  x_1-p_1, 0, 0  \ldots) \in \Zm.
\end{align}

\begin{proposition}\label{prop.ippin1}
	Let $\hx = ( \ldots, x_2, x_1) \in \imp$ be such that  $x_m \leq p_m$ for all $m \in \Z_{\geq 1}$,
	and let $l \geq 1$. The following are equivalent:
	\begin{enumerate}[\upshape(1)]
		\item $z_1(\hx, l) \in \imp$;\label{enu.lem.ippin1}
		\item $c_j x_{j-2l} -c_{j-1} x_{j-2l+1} \geq 0$ for $j \geq 2l+1 $;\label{enu.lem.ippin2}
		\item $0=\ep_{i_j}(\f_{i_{j+1}}^{p_{j+1}}\cdots \f_{i_{-1}}^{p_{-1}} \f_{i_{0}}^{p_{0}} \hx^\ast )$ for  $-2l+1 \leq j \leq 0$.
		 	\label{enu.lem.ippine1}
	\end{enumerate}
\end{proposition}

\begin{proof}
	\eqref{enu.lem.ippin1} $\Leftrightarrow$ \eqref{enu.lem.ippin2}:
	By Proposition \ref{prop.impsi},
	we see that
	$z_1(\hx, l)  \in \imp $ if and only if
	\begin{alignat}{2}
		\label{eq.zyl1}
		&c_j p_{j-2l} -c_{j-1} p_{j-2l+1} \geq 0  & \quad &  \text{ for } 1 \leq j \leq 2l-1,\\
		\label{eq.zyl2}
		&c_j p_{j-2l} -c_{j-1} x_{j-2l+1} \geq 0  & \quad &  \text{ for } j= 2l,\\
		\label{eq.zyl3}
		&c_j x_{j-2l} -c_{j-1} x_{j-2l+1} \geq 0  & \quad &  \text{ for } 2l+1 \leq j.
	\end{alignat}
	Therefore it is obvious that \eqref{enu.lem.ippin1} implies \eqref{enu.lem.ippin2}.
	Assume that  \eqref{enu.lem.ippin2} holds;
	we need to show that \eqref{eq.zyl1} and \eqref{eq.zyl2}.
	We can easily see by induction on $j$ that
	\begin{equation}\label{eq.cp1}
		c_j p_{j-2l} -c_{j-1} p_{j-2l+1} \geq 0 \quad \text{ for }    j\geq1.
	\end{equation}
	In particular, we get \eqref{eq.zyl1}.
	Since $x_1 \leq p_1$, we see that
	$c_{2l} p_{0} -c_{2l-1} x_{1} \geq c_{2l} p_{0} -c_{2l-1} p_{1} $.
	Combining this inequality and \eqref{eq.cp1}, we obtain \eqref{eq.zyl2}.

	\eqref{enu.lem.ippin1} $\Leftrightarrow$ \eqref{enu.lem.ippine1}:
	By Proposition \ref{prop.element}, together with the fact
	that $i_s=i_t$ if $s \equiv t \bmod 2$,
	we see that  $z_1(\hx, l)  \in \imp $ if and only if
	\begin{alignat}{2}
		& 0=\ep_{i_j}(\f_{i_{j+1}}^{p_{j+1}} \f_{i_{j+2}}^{p_{j+2}}
		\cdots \f_{i_{-1}}^{p_{-1}} \f_{i_{0}}^{p_{0}} \f_{i_{1}}^{x_1} \f_{i_{2}}^{x_2}\cdots z_\infty )
		& \quad & \text{ for } -2l+1 \leq j \leq 0,\label{eq.epzl1}\\
		& 0=\ep_{i_{j}}(\f_{i_{j+1}}^{x_{j+1}} \f_{i_{j+2}}^{x_{j+2}} \cdots z_\infty )  & \quad & \text{ for }   j \geq 1.	\label{eq.epzl2}
	\end{alignat}
	Since $\hx \in \imp$,  we have 	$\hx^\ast=\f_{i_1}^{x_1} \f_{i_2}^{x_2}\cdots z_\infty$ and
	$0=\ep_{i_{j}}(\f_{i_{j+1}}^{x_{j+1}} \f_{i_{j+2}}^{x_{j+2}} \cdots z_\infty ) $ for  $j \geq 1$.
	Therefore \eqref{enu.lem.ippin1}  is equivalent to  \eqref{enu.lem.ippine1}.

	Thus we have proved the proposition.
\end{proof}

By using  Propositions \ref{prop.element} and \ref{prop.impsi},
together with the fact that
$i_s = i_t $ if $ s \equiv t \bmod 2$, we can prove the following proposition
in exactly the same way as Proposition \ref{prop.ippin1}.

\begin{proposition}\label{prop.ippin2}
	Let $\hx = ( \ldots, x_2, x_1) \in \imp$ be such that  $x_m \leq p_m$ for all $m \in \Z_{\geq 1}$, and let $l \geq 1$.
	The following are equivalent:
	\begin{enumerate}[\upshape(1)]
		\item $z_2(\hx, l) \in \imm$;\label{enu.lem.ippin4}
		\item $c'_{-j+1} (x_{j+2l}-p_{j+2l}) -c'_{-j} (x_{j+2l-1}-p_{j+2l-1}) \leq 0$  for  $-2l+2 \leq j \leq -1$; \label{enu.lem.ippin5}
		\item $ 0= \ph_{i_{j}}(\e_{i_{j-1}}^{p_{j-1}-x_{j-1}}  \cdots \e_{i_2}^{p_2-x_2} \e_{i_1}^{p_1-x_1}z_{-\infty})$ for $1 \leq j \leq  2l$.\label{enu.lem.ippin7}
	\end{enumerate}
\end{proposition}

\begin{proposition}\label{prop.cxga1}
	Let $\hx = ( \ldots, x_2, x_1) \in \imp$ be such that  $x_m \leq p_m$ for all $m \in \Z_{\geq 1}$, and let $k \geq 1$.
	The following are equivalent:
	\begin{enumerate}[\upshape(1)]
		\item $c_{k+2l} x_{k} -c_{k+2l-1} x_{k+1} \geq 0$ for $l \geq 1$;\label{enu.lem.cxga1}
		\item $\ga_{k}  x_{k}-x_{k+1}\geq 0$.\label{enu.lem.cxga2}
	\end{enumerate}
\end{proposition}

\begin{proof}
	Assume that \eqref{enu.lem.cxga2} holds.
	By Lemma \ref{lem.cab}, together with \eqref{eq.gamma},
	we have $c_{k+2l} > \ga_{k} c_{k+2l-1} $ for $l\geq 1$.
	Hence we obtain
	$c_{k+2l} x_{k} -c_{k+2l-1} x_{k+1} \geq c_{k+2l-1} (\ga_{k} x_{k}- x_{k+1})$.
	By the assumption, we obtain \eqref{enu.lem.cxga1}.

	Assume that \eqref{enu.lem.cxga1} holds;
	note that $x_k \geq 0$.
	If $x_k=0$, then we have $-c_{k+2l-1} x_{k+1} \geq 0$.
	Since $c_{k+2l-1} >0$, we see that $x_{k+1}=0$,
	which gives $\ga_{k}  x_{k}-x_{k+1} = 0$.
	Assume that $x_k > 0$.
	By the assumption,
	we obtain $c_{k+2l}/c_{k+2l-1} \geq x_{k+1}/x_k$ for $l \geq 1$.
	Since the sequence $\{c_{k+l}/c_{k+2l-1}\}_{l\geq 1}$ is strictly decreasing,
	and converges to $\ga_k$ by Lemma \ref{lem.cab},
	we see that  ${x_{k+1}}/{x_{k}} \leq \ga_k$, which is equivalent to \eqref{enu.lem.cxga2}.
\end{proof}

\begin{proposition}\label{prop.cxga2}
	Let $\hx = ( \ldots, x_2, x_1) \in \imp$ be such that  $x_m \leq p_m$ for all $m \in \Z_{\geq 1}$,
	and let $k \geq 1$.
	The following are equivalent:
	\begin{enumerate}[\upshape(1)]
		\item $c'_{2l+i_k} (x_{k+1}-p_{k+1}) -c'_{2l+i_k-1} (x_{k}-p_{k}) \leq 0$  for  $l\geq 1$;\label{enu.lem.cxga4}
		\item $\ga_{k+1} p_{k+1} -p_{k} +x_{k} - \ga_{k+1} x_{k+1} \geq 0 $.\label{enu.lem.cxga3}
	\end{enumerate}
\end{proposition}

\begin{proof}
	Assume that \eqref{enu.lem.cxga3} holds.
	By Lemma \ref{lem.cab}, we have
	$ c'_{2l+i_k} > \ga_{k+1} c'_{2l+i_k-1}  $ for $l\geq 1$.
	Since $x_{k+1}-p_{k+1}\leq 0$,  we see that
	\begin{equation}
		c'_{2l+i_k} (x_{k+1}-p_{k+1}) -c'_{2l+i_k-1} (x_{k}-p_{k})
		\leq 	\underbrace{c'_{2l+i_k-1}}_{>0} (\underbrace{\ga_{k+1} x_{k+1}-\ga_{k+1} p_{k+1}-x_{k}+p_{k}}_{\leq 0 \text{ by assumption}}) \leq 0.
	\end{equation}

	Assume that \eqref{enu.lem.cxga4} holds;
	note that $(0 \leq ) \: x_{k+1} \leq p_{k+1}$.
	If $x_{k+1}=p_{k+1}$, then we have $-c'_{2l+i_k-1} (x_{k}-p_{k})\leq 0$.
	Since $c'_{2l+i_k-1} > 0$, we obtain  $x_{k}=p_{k}$,
	which gives
	$\ga_{k+1} p_{k+1} -p_{k} +x_{k} - \ga_{k+1} x_{k+1} = 0$.
	Assume that $x_{k+1} < p_{k+1}$.
	By the assumption,
	we obtain $c'_{2l+i_k}/c'_{2l+i_k-1} \geq (x_{k}-p_{k})/(x_{k+1}-p_{k+1})$.
	Since the sequence
	$\{c'_{2l+i_k}/c'_{2l+i_k-1}\}_{l\geq 1 }$
	is decreasing
	and converges to $\ga_{k+1}$ by Lemma \ref{lem.cab},
	we obtain $(x_{k}-p_{k})/(x_{k+1}-p_{k+1}) \leq \ga_{k+1}$, which is equivalent to \eqref{enu.lem.cxga3}.
\end{proof}

By Propositions \ref{prop.ippin1} --  \ref{prop.cxga2}, we obtain the following corollary.
\begin{corollary}\label{cor.equiv}
	Let $\hx = ( \ldots, x_2, x_1) \in \imp$
	be such that  $x_m \leq p_m$ for all $m \in \Z_{\geq 1}$.
	\begin{enumerate}[\upshape(1)]
		\item 	$\ga_{k}  x_{k}-x_{k+1}\geq 0$ for all $k \geq 1$
			if and only if
			\begin{equation}
				\ep_{i_j}(\f_{i_{j+1}}^{p_{j+1}}\cdots \f_{i_{-1}}^{p_{-1}} \f_{i_{0}}^{p_{0}} \hx^\ast ) =0
			\end{equation}
			for all  $j \leq 0$.
		\item  $\ga_{k+1} p_{k+1} -p_{k} +x_{k} - \ga_{k+1} x_{k+1} \geq 0$ for all $k \geq 1$
			if and only if
			\begin{equation}
				\ph_{i_{j}}(\e_{i_{j-1}}^{p_{j-1}-x_{j-1}}  \cdots \e_{i_2}^{p_2-x_2} \e_{i_1}^{p_1-x_1}z_{-\infty}) = 0
			\end{equation}
			for all $j \geq 1$.
	\end{enumerate}
\end{corollary}

\subsection{Proof of Theorem \ref{thm.silsc}.}\label{seq.prf.thm.silsc}
\begin{lemma}\label{lem.silinipl}
	It holds that  $\Sil \subset \ipl$.
\end{lemma}

\begin{proof}
	Let $\vx= ( \ldots, x_2, x_1) \otimes t_\lambda \otimes (  x_0, x_{-1}, \ldots) \in \Sil$.
	By Proposition \ref{prop.impsi},
	it suffices to show that
	\begin{alignat}{2}
		&c_j x_j -c_{j-1}x_{j+1}\geq 0 & \quad & \text{ for }  j \geq 1,\label{eq.silnip1}\\
		&c'_{-j+1} x_j -c'_{-j}x_{j-1} \leq 0 & \quad & \text{ for }  j \leq 0.\label{eq.silnip2}
	\end{alignat}

	First, we verify  \eqref{eq.silnip1}.
	If $j=1$, then the assertion is obvious because $c_j=1$ and $c_{j-1}=0$.
	Assume that $j>1$; note that $c_{j-1}>0$.
	It follows from Lemma \ref{lem.cab} that  $\ga_j <c_{j}/c_{j-1}$.
	Also, we have $(\ga_{j}  \ze_{j}-\ze_{j+1})(\vx)=\ga_{j}x_j-x_{j+1} \geq0$
	by the definition of $\Sil$. Hence
	\begin{equation}
		c_j x_j -c_{j-1}x_{j+1}=c_{j-1}\left(\frac{c_{j}}{c_{j-1}}x_j-x_{j+1}\right)
		\geq c_{j-1}(\ga_{j}x_j-x_{j+1}) \geq 0.
	\end{equation}

	Next, we verify \eqref{eq.silnip2}.
	If $j=0$, then the assertion is obvious because $c'_{-j+1}=1$ and $c'_{-j}=0$.
	Assume that $j<0$; note that $c'_{-j}>0$.
	It follows from Lemma \ref{lem.cab} that  $\ga_j < c'_{-j+1}/c'_{-j}$.
	Also, we have $(\ze_{j-1}- \ga_{j} \ze_{j})(\vx)= x_{j-1}- \ga_{j} x_{j} \geq 0$
	by the definition of $\Sil$. Hence
	\begin{equation}
		c'_{-j+1} x_j -c'_{-j}x_{j-1}
		=c'_{-j} \left(\frac{c'_{-j+1}}{c'_{-j}} x_j -  x_{j-1} \right)
		\leq c'_{-j}(\ga_{j} x_{j} - x_{j-1})
		\leq 0.
	\end{equation}

	Thus we have proved the lemma.
\end{proof}

For $k\in\Z$,
we set $k^{(+)}:=k + 2$ and $k^{(-)}:=k - 2$.
Also, we define the function $\barb_{k}: \R^{\infty} \to \R$ by
\begin{equation}
	\barb_k=
	\begin{cases*}
		-\pair{\la}{\alc_{i_k}} + \ze_k-a_{i_k} \ze_{k+1} +\ze_{k+2} & if $k = -1,0$,\\
		\ze_k-a_{i_k} \ze_{k+1} +\ze_{k+2} & otherwise;
	\end{cases*}
\end{equation}
note that $\barb_k(\vx)=\sigma_k(\vx)-\sigma_{k^{(+)}}(\vx)$.
Moreover,
for $k \in \Z$, we define the operator $\bars_k$
on  $\{ c+\sum_{l\in\Z} \phi_l \ze_l \mid c, \phi_l \in\R \}$
as follows: for $\phi= c+\sum_{l\in\Z} \phi_l \ze_l$ with $c, \phi_l \in\R$, we set
\begin{equation}
		\bars_k(\phi) :=
		\begin{cases*}
			\phi-\phi_k \barb_{k^{(+)}} & if $\phi_k  \geq 0$, \\
			\phi-\phi_k \barb_{k^{(-)}} & if $\phi_k  < 0$;
		\end{cases*}
\end{equation}
note that $\bars_k(\phi)=\phi$ if $\phi_k=0$.

\begin{lemma}\label{lem.sksub}
	Let $\Xi$ be a subset of $\{ c+\sum_{l\in\Z} \phi_l \ze_l \mid c, \phi_l \in\R \}$.
	Assume that
	\begin{equation}\label{eq.skpos}
		\bars_k(\phi) \in  \sum_{j\geq 1} \R_{\geq 0}\ze_j + \sum_{j\leq 0} \R_{\geq 0}(-\ze_j )+ \sum_{\psi \in \Xi} \R_{\geq 0} \psi
	\end{equation}
	for all $\phi \in \Xi$ and  $k \in \Z$.
	Then, $\Sigma =\{\vx \in \Zil \mid \phi(\vx) \geq 0 \text{ for all } \phi \in \Xi\}$ 	is a subcrystal of  $\Zil$.
\end{lemma}

\begin{proof}
	This lemma can be shown similarly to \cite[Lemma 4.3]{HN}.
	Let $\vx \in \Sigma$.
	We show that  if $\f_i\vx\neq \0$, then  $\f_i\vx \in \Sigma$,
	that is, $\phi(\f_i\vx) \geq 0$ for all $\phi \in \Xi$.
	Let us write $\phi= c+\sum_{l\in\Z} \phi_l \ze_l$ with $c, \phi_l \in\R$.
	Define $M_{(i)}=M_{(i)}(\vx)$ as \eqref{eq.mi}, and set $k:=\min M_{(i)}$.
	We see by \eqref{eq.deffi}
	that $\phi(\f_i\vx) = \phi(\vx)+\phi_k$.
	If $\phi_k \geq 0$, then
	the assertion is  obvious
	because $\phi(\f_i\vx) = \phi(\vx)+\phi_k \geq \phi(\vx) \geq 0$.
	Assume that  $\phi_k < 0$.
	By the definition of $M_{(i)}$ and the fact that $i_k=i_m$ if  $k\equiv m \bmod 2$, we have $\sigma_{k}(\vx)>  \sigma_{k-2n}(\vx)$ for all $n\in\Z_{\geq 1}$.
	In particular, $\sigma_{k}(\vx) > \sigma_{k^{(-)}}(\vx)$.
	Since $\barb_{k^{(-)}}(\vx)=\sigma_{k^{(-)}}(\vx)-\sigma_{k}(\vx) \in \Z$,
	we deduce that  $\barb_{k^{(-)}}(\vx)\leq -1$.
	It follows that
	\begin{equation}
		\phi(\f_i\vx)
		= \phi(\vx)+\phi_k
		\geq  \phi(\vx) - \phi_k \barb_{k^{(-)}}(\vx)
		=(\bars_k(\phi))(\vx).
	\end{equation}
	By assumption \eqref{eq.skpos}, we see that $\bars_k(\phi)$ is of the form  $\bars_k(\phi) = \sum_{j\geq 1} t_{j}\ze_j + \sum_{j\leq 0} t_{j}(-\ze_j )+ \sum_{\psi \in \Xi} t_{\psi} \psi$,
	where $ t_{j}, t_{\psi} \in \R_{\geq 0}$.
	Since $\vx \in \Sigma$, we have $\psi(\vx) \geq 0 \text{ for any } \psi \in \Xi$.
	Therefore we see that
	\begin{equation}
		\phi(\f_i\vx)\geq (\bars_k(\phi))(\vx) =
 		\sum_{j\geq 1} t_j \underbrace{x_j}_{\geq 0} + \sum_{j\leq 0} t_j (\underbrace{-x_j}_{\geq0})+ \sum_{\psi \in \Xi} t_\psi \underbrace{\psi(\vx)}_{\geq 0}
		\geq 0.
	\end{equation}
	Thus we get $\f_i\vx \in \Sigma$.
	Similarly, we can show that $\e_i\vx \in \Sigma$ if $\e_i\vx\neq \0$.

	Thus we have proved the lemma.
\end{proof}

\begin{proof}[Proof of Theorem \ref{thm.silsc}]
By Lemmas \ref{lem.silinipl} and  \ref{lem.sksub}, it suffices to show that
\begin{equation}\label{eq.skphi}
	\bars_k(\phi) \in  \sum_{j\geq 1} \R_{\geq 0}\ze_j + \sum_{j\leq 0} \R_{\geq 0}(-\ze_j )+ \sum_{\psi \in \Xil} \R_{\geq 0} \psi
\end{equation}
for all $k \in \Z$ and $\phi \in \Xil$.
Here we verify \eqref{eq.skphi}  for the case that  $\phi = \ga_0 p_0 +\ga_0 \ze_{0}- \ze_{1}$; for the other cases, see Appendix \ref{apd.1}.
If $k\neq 0, 1$, then the assertion is trivial since $\bars_k(\phi)=\phi$.
Assume that $k=0$. We compute
\begin{align}
	\bars_0(\phi )
	&=(\ga_0 p_0 +\ga_0 \ze_{0}- \ze_{1})-\ga_0 \barb_{0} \\
	&=(\ga_0 p_0 +\ga_0 \ze_{0}- \ze_{1})-\ga_0 (p_0+\ze_0-a_2 \ze_1 +\ze_2) \\
	&=\ga_0\left(  \left(a_{i_0} - \frac{1}{\ga_0}  \right)\ze_1 -\ze_2 \right)\\
	&=\ga_0 ( \underbrace{\ga_1 \ze_1-\ze_2}_{\in \Xil} ) \quad \text{ by \eqref{eq.gamaik}}.
\end{align}
Assume that $k=1$. We compute
\begin{align}
	\bars_1(\phi )
	&=(\ga_0 p_0 +\ga_0 \ze_{0}- \ze_{1})-(-1)\barb_{-1} \\
	&=(\ga_0 p_0 +\ga_0 \ze_{0}- \ze_{1})+(-p_1+\ze_{-1}-a_{-1} \ze_{0} +\ze_1) \\
	&=\ga_0 p_0 -p_1+\ze_{-1}+ ( \ga_0 -a_{-1} ) \ze_{0} \\
	&=\ga_0 p_0 -p_1+\ze_{-1}- \frac{1}{\ga_{-1}} \ze_{0} \quad \text{ by  \eqref{eq.gamaik}} \\
	&\begin{aligned}
		=-p_{-1}+ \left(\ga_0 + \frac{1}{\ga_1}  \right) p_0 -p_1+
		\frac{1}{\ga_{-1}} (\ga_{-1}p_{-1} - p_0 & +\ga_{-1} \ze_{-1} -\ze_0)\\
		&(\text{note that }  \ga_{-1}=\ga_1 )
		\end{aligned}
		\\
	&=-p_{-1}+ a_1 p_0 -p_1+
		\frac{1}{\ga_{-1}} (\ga_{-1}p_{-1} -p_0 +\ga_{-1} \ze_{-1} -\ze_0) \quad \text{ by \eqref{eq.gamaik}}\\
	&=0+ \frac{1}{\ga_{-1}} (\underbrace{\ga_{-1}p_{-1} -p_0 +\ga_{-1} \ze_{-1} -\ze_0}_{\in \Xil})
	\quad \text{ by \eqref{eq.pm2}}.
\end{align}
Thus we have proved Theorem \ref{thm.silsc}.
\end{proof}

\subsection{Proof of Theorem \ref{prop.silex}.}\label{seq.prf.prop.silex}
Let $\vz=z_1\otimes t_\lambda \otimes z_2 \in \ipl$.
First, by the tensor product rule of  crystals (see also \cite[Appendix B]{Klvz}),
we see that
	\begin{align}
				\ep_i(\vz)&=\max\{ \ep_i(z_1), \ph_i(z_2) -\pair{\wt(\vz)}{\alc_i} \},\label{eq.tnsrep}\\
				\ph_i(\vz)
				&=\max\{ \ep_i(z_1)+\pair{\wt(\vz)}{\alc_i}, \ph_i(z_2) \}\label{eq.tnsrph}.
			\end{align}
Moreover,
\begin{equation}\label{eq.eic}
	\e_i^{\ep_i(\vz)} \vz = \e_i^{\ep_i(z_1)}z_1 \otimes t_\lambda \otimes \e_i^c z_2,
\end{equation}
where $c= \max\{ -\ep_i(z_1)+\ph_i(z_2)-\pair{\wt(\vz)}{\alc_i}, 0 \}$.
Next, for $k \in \Z$, we set
\begin{equation}
w_{k}:=
	\begin{cases*}
		(s_2 s_1)^n & if   $k=2n$  with $n \in \Z_{\geq 0}$, \\
		s_1(s_2 s_1)^n & if  $k=2n+1$  with $n \in \Z_{\geq 0}$, \\
		(s_1 s_2)^{-n} & if  $k=2n$  with $n \in \Z_{\leq 0}$, \\
		s_2(s_1 s_2)^{-n} & if  $k=2n-1$  with $n \in \Z_{\leq 0}$;
	\end{cases*}
\end{equation}
note that $W=\{ w_k \mid k \in \Z \}$.
By \cite[Lemma 3.3]{H}, we have
\begin{equation}\label{eq.xm}
	w_k\lambda=
	\begin{cases*}
		p_{k+1}\Lambda_1-p_k\Lambda_2 &
		if  $k$   is even,\\
		-p_k\Lambda_1+p_{k+1}\Lambda_2 &
		if  $k$   is odd
	\end{cases*}
\end{equation}
for $k \in \Z$.
Since $\wt(S_{w_k} \vz^\ast ) = w_k \wt(\vz^\ast)= - w_k  \lambda$, we see that
\begin{equation}\label{eq.pairsxalc}
	\pair{\wt(S_{w_k} \vz^\ast )}{\alc_i} =
	\begin{cases*}
		p_{k}  & if  $i=i_k$,\\
		-p_{k+1}  & if $i=i_{k+1}$,
	\end{cases*}
\end{equation}
and hence
\begin{equation}\label{eq.sw}
	S_{w_k}\vz^\ast
	=	\e_{i_k}^{p_{k}} S_{w_{k-1}}\vz^\ast
	=	\f_{i_{k+1}}^{p_{k+1}} S_{w_{k+1}}\vz^\ast.
\end{equation}

\begin{proposition}\label{prop.sx}
	Let  $\vx= \hx \otimes t_\lambda \otimes z_{-\infty} \in \Sil'$ with $\hx=(\ldots,  x_2,  x_1)$. Then,
	\begin{equation}
		S_{w_k}\vx^\ast=
		\begin{cases*}
			\e_{i_k}^{x_k}\cdots \e_{i_2}^{x_2} \e_{i_1}^{x_1}\hx^\ast \otimes t_\mu \otimes \e_{i_k}^{p_k-x_k}  \cdots \e_{i_2}^{p_2-x_2} \e_{i_1}^{p_1-x_1}z_{-\infty}
			& {\rm{ if }} $k \geq 0$, \\
			\f_{i_{k+1}}^{p_{k+1}}  \cdots \f_{i_{-1}}^{p_{-1}} \f_{i_0}^{p_0}\hx^\ast \otimes t_\mu \otimes z_{-\infty}
			& {\rm{ if }}  $k \leq 0$, \\
		\end{cases*}
	\end{equation}
	where $\mu:=-\lambda-\wt(\hx)$.
\end{proposition}

\begin{proof}
	Since $\hx \in \imp$ by Lemma \ref{lem.silinipl}, it follows from  Proposition \ref{prop.element} that
	\begin{equation}
		\label{eq.1vyastep} x_j=\ep_{i_j}(\e_{i_{j-1}}^{x_{j-1}}\cdots \e_{i_2}^{x_{2}} \e_{i_1}^{x_{1}} \hx^\ast ) \quad \text{ for } j \geq 1.
	\end{equation}
	By the definition of $\Sil'$, we have
	$  p_k - x_k\geq 0$,
	$\ga_{k}  x_{k}-x_{k+1} \geq 0 $, and
	$\ga_{k+1} p_{k+1} -p_{k} +x_{k} - \ga_{k+1} x_{k+1} \geq 0 $ for all $k\geq 1$.
	By Corollary \ref{cor.equiv}, we see that
	\begin{equation}\label{eq.mainlem1}
		\ep_{i_j} (\f_{i_{j+1}}^{p_{j+1}}\cdots \f_{i_{-1}}^{p_{-1}} \f_{i_0}^{p_0}\hx^\ast)=0
		\quad \text{ for }
 		j \leq 0,
	\end{equation}
	\begin{equation}\label{eq.mainleme1}
		\ph_{i_{j}}(\e_{i_{j-1}}^{p_{j-1}-x_{j-1}} \cdots \e_{2}^{p_{2}-x_{2}}  \e_{1}^{p_{1}-x_{1}} z_{-\infty} )=0 \quad \text{ for }
		j \geq 1.
	\end{equation}

	Now, we show the assertion by induction on $|k|$.
	If $k=0$, then the assertion is obvious by \eqref{eq.zast}.
	Assume that $k \geq 1$.
	% By \eqref{eq.sw}, we have
	% $S_{w_{k}}\vx^\ast = \e_{i_k}^{p_k}S_{w_{k-1}}\vx^\ast$.
	By the induction hypothesis, we obtain
	\begin{equation}
		S_{w_{k-1}}\vx^\ast=\e_{i_{k-1}}^{x_{k-1}} \cdots \e_{i_2}^{x_2} \e_{i_1}^{x_1}\hx^\ast \otimes t_\mu \otimes \e_{i_{k-1}}^{p_{k-1}-x_{k-1}}  \cdots \e_{i_2}^{p_2-x_2} \e_{i_1}^{p_1-x_1}z_{-\infty}.
	\end{equation}
	We have $\pair{\wt(S_{w_{k-1}}\vx^\ast)}{\alc_{i_k}} = - p_k\leq 0$ by \eqref{eq.pairsxalc},
	$\ep_{i_k}( \e_{i_{k-1}}^{x_{k-1}} \cdots \e_{i_2}^{x_2} \e_{i_1}^{x_1}\hx^\ast) =x_k$
	by \eqref{eq.1vyastep},
	and $\ph_{i_k}( \e_{i_{k-1}}^{p_{k-1}-x_{k-1}}  \cdots \e_{i_2}^{p_2-x_2} \e_{i_1}^{p_1-x_1}z_{-\infty})=0$ by \eqref{eq.mainleme1}.
	Since $x_k \leq p_k$ as seen above,
	% it follows from Lemma \ref{lem.zef} \eqref{lem.zef1} that
	we see by \eqref{eq.eic} and \eqref{eq.sw} that
	\begin{align}
		S_{w_k}\vx^\ast
		& = \e_{i_k}^{p_k}(\e_{i_{k-1}}^{x_{k-1}} \cdots \e_{i_2}^{x_2} \e_{i_1}^{x_1}\hx^\ast \otimes t_\mu \otimes \e_{i_{k-1}}^{p_{k-1}-x_{k-1}}  \cdots \e_{i_2}^{p_2-x_2} \e_{i_1}^{p_1-x_1}z_{-\infty})\\
		&=  \e_{i_k}^{x_k} \e_{i_{k-1}}^{x_{k-1}}\cdots \e_{i_2}^{x_2} \e_{i_1}^{x_1}\hx^\ast \otimes t_\mu \otimes \e_{i_{k}}^{p_{k}-x_{k}} \e_{i_{k-1}}^{p_{k-1}-x_{k-1}} \cdots \e_{i_2}^{p_2-x_2} \e_{i_1}^{p_1-x_1}z_{-\infty}.
	\end{align}
	Assume that $k \leq -1$.
	By the induction hypothesis, we obtain
	\begin{equation}
		S_{w_{k+1}}\vx^\ast= \f_{i_{k+2}}^{p_{k+2}}  \cdots \f_{i_{-1}}^{p_{-1}} \f_{i_0}^{p_0}\hx^\ast \otimes t_\mu \otimes z_{-\infty}.
	\end{equation}
	Since $S_{w_{k}}\vx^\ast \neq \0$,
	we see by \eqref{eq.sw} that
	\begin{equation}
		S_{w_{k}}\vx^\ast
		= \f_{i_{k+1}}^{p_{k+1}}(\f_{i_{k+2}}^{p_{k+2}}  \cdots \f_{i_{-1}}^{p_{-1}} \f_{i_0}^{p_0}\hx^\ast \otimes t_\mu \otimes z_{-\infty})
		 = \f_{i_{k+1}}^{p_{k+1}}\f_{i_{k+2}}^{p_{k+2}}  \cdots \f_{i_{-1}}^{p_{-1}} \f_{i_0}^{p_0}\hx^\ast \otimes t_\mu \otimes z_{-\infty}.
	\end{equation}
	Thus we have proved the proposition.
\end{proof}

\begin{proof}[Proof of Theorem \ref{prop.silex}]
	Keep the notation and setting in Proposition \ref{prop.sx}.
	We show that $\vx^\ast $ is extremal;
	by \eqref{eq.pairsxalc},
	it suffices to show that
	$\ep_{i_k} (S_{w_k} \vx^\ast )=0$ and  $\ph_{i_{k+1}} (S_{w_k} \vx^\ast) =0$
	 for all  $k \in \Z$.

	{\bf Step 1}.
	Assume that $k\geq 0$. We show that $\ph_{i_{k+1}} (S_{w_k} \vx^\ast) =0$.
	We know from Proposition \ref{prop.sx} that
	\begin{equation}
		S_{w_k}\vx^\ast=
			\e_{i_k}^{x_k} \cdots \e_{i_2}^{x_2} \e_{i_1}^{x_1}\hx^\ast \otimes t_\mu \otimes \e_{i_k}^{p_k-x_k}  \cdots \e_{i_2}^{p_2-x_2} \e_{i_1}^{p_1-x_1}z_{-\infty}.
	\end{equation}
	By the same argument as in the proof of Proposition \ref{prop.sx},
	we see that
	$\pair{\wt(S_{w_{k}}\vx^\ast)}{\alc_{i_{k+1}}} = - p_{k+1} \leq 0$,
	$\ep_{i_{k+1}} (\e_{i_k}^{x_k}\cdots \e_{i_2}^{x_2} \e_{i_1}^{x_1}\hx^\ast)=x_{k+1}$,
	$\ph_{i_{k+1}}(\e_{i_k}^{p_k-x_k}  \cdots \e_{i_2}^{p_2-x_2} \e_{i_1}^{p_1-x_1}z_{-\infty})=0$,
	and  $x_{k+1} \leq p_{k+1}$.
	Thus, by \eqref{eq.tnsrph},
	$\ph_{i_{k+1}}(S_{w_k}\vx^\ast) = \max\{ x_{k+1}+(-p_{k+1}), 0 \} = 0 $.

	{\bf Step 2}.
	Assume that $k>0$. We show  that $\ep_{i_{k}} (S_{w_k} \vx^\ast) =0$.
	We have
	\begin{align}
		\ep_{i_k}(S_{w_{k}} \vx^\ast)
		&= \ep_{i_k}(\e_{i_k}^{p_k} S_{w_{k-1}} \vx^\ast)
		=\ep_{i_k}( S_{w_{k-1}} \vx^\ast) -p_k\\
		&=\ph_{i_k}( S_{w_{k-1}} \vx^\ast)-
		\underbrace{\pair{\wt( S_{w_{k-1}} \vx^\ast)}{\alc_{i_k}}}_{=-p_k  \text{ by \eqref{eq.pairsxalc}}}
		 -p_{k}
		=\ph_{i_k}( S_{w_{k-1}} \vx^\ast).
	\end{align}
	Since $\ph_{i_k}( S_{w_{k-1}} \vx^\ast)=0$ by Step 1, we obtain $\ep_{i_k}(S_{w_{k}} \vx^\ast)=0$.

	{\bf Step 3}.
	Assume that $k \leq 0$. We show that $\ep_{i_k}(S_{w_{k}} \vx^\ast) =0$.
	We know from Proposition \ref{prop.sx} that
	\begin{equation}
		S_{w_k}\vx^\ast=
			\f_{i_{k+1}}^{p_{k+1}}  \cdots \f_{i_{-1}}^{p_{-1}} \f_{i_0}^{p_0}\hx^\ast \otimes t_\mu \otimes z_{-\infty}.
	\end{equation}
	We have
	$\pair{\wt(S_{w_{k}}\vx^\ast)}{\alc_{i_{k}}} = p_{k}$
	by \eqref{eq.pairsxalc} and
	$\ep_{i_{k}} (\f_{i_{k+1}}^{p_{k+1}}  \cdots \f_{i_{-1}}^{p_{-1}} \f_{i_0}^{p_0}\hx^\ast)=0$
	by \eqref{eq.mainlem1}.
	Since $\ph_{i_k}(z_{-\infty})=0$,
	we see by \eqref{eq.tnsrep} that
	$\ep_{i_k}(S_{w_k}\vx^\ast) = \max\{ 0,  0-p_k \} = 0 $.

	{\bf Step 4}.
	Assume that $k<0$. We show  that $\ph_{i_{k+1}} (S_{w_k} \vx^\ast) =0$.
	We have
	\begin{align}
		\ph_{i_{k+1}}(S_{w_{k}} \vx^\ast)
		&= \ph_{i_{k+1}}(\f_{i_{k+1}}^{p_{k+1}} S_{w_{k+1}} \vx^\ast)
		=\ph_{i_{k+1}}( S_{w_{k+1}} \vx^\ast) -p_{k+1}\\
		&=\ep_{i_{k+1}}( S_{w_{k+1}} \vx^\ast)+
		\underbrace{\pair{\wt( S_{w_{k+1}} \vx^\ast)}{\alc_{i_{k+1}}}}_{=p_{k+1}   \text{ by \eqref{eq.pairsxalc}}} -p_{k+1}
		=\ep_{i_{k+1}}( S_{w_{k+1}} \vx^\ast).
	\end{align}
	Since $\ep_{i_{k+1}}(S_{w_{k+1}} \vx^\ast) =0$ by Step 3,
	we obtain $\ph_{i_{k+1}}(S_{w_{k}} \vx^\ast)=0$.

	This completes the proof of Theorem \ref{prop.silex}.
\end{proof}

\subsection{Proof of Theorem \ref{prop.silcon}.}\label{seq.prf.prop.silcon}
Let $\vy = \hy_1 \otimes t_\la \otimes \hy_2 \in \ipl$ with $\hy_1 \in \imp$ and
$\hy_2 \in \imm$,
and assume that $\vy^\ast$ is extremal.
Since $\imm \cong \CB(-\infty)$ as crystals, there exist $i_1, \ldots, i_l$ such that
$\f_{i_l}^{\mathrm{max}} \cdots \f_{i_1}^{\mathrm{max}} \hy_2 = z_{-\infty}$.
By the tensor product rule of crystals,
if we set
$\vx:= \f_{i_l}^{\mathrm{max}} \cdots \f_{i_1}^{\mathrm{max}}  \vy$,
then $\vx$ is of the form
$ \vx= \hx \otimes t_\lambda \otimes z_{-\infty}$ with $\hx \in \imp$;
in order to prove Theorem \ref{prop.silcon}, it suffices to show that
$\vx\in \Sil$.
Let us write   $\hx = (\ldots, x_2, x_1)$.
By the definition of $\Sil$, we deduce that $\vx= (\ldots, x_2, x_1) \otimes t_\lambda \otimes z_{-\infty}\in \Sil$ if and only if
\begin{alignat}{2}
&p_{k} - x_{k} \geq 0 &\quad& \text{ for } k \geq 1;   \label{eq.exsinplyzil1} \\
&\ga_{k}  x_{k}-x_{k+1} \geq 0 &\quad& \text{ for } k \geq 1;      \label{eq.exsinplyzil2} \\
&\ga_{k+1} p_{k+1} -p_{k} +x_{k} - \ga_{k+1} x_{k+1} \geq 0 &\quad& \text{ for } k \geq 1; \label{eq.exsinplyzil3}\\
&\ga_0 p_{0} + \ga_0 \cdot 0 - x_{1} \geq 0;   \label{eq.exsinplyzil4} \\
&\ga_1 p_{1}  +  0 - \ga_1 x_{1} \geq 0.   \label{eq.exsinplyzil5}
\end{alignat}
Assume that \eqref{eq.exsinplyzil1} holds.
Then it is obvious that  \eqref{eq.exsinplyzil5} holds.
Moreover, we obtain
$\ga_0 p_{0} + \ga_0 \cdot 0 - x_{1}  \geq \ga_0 p_{0} - p_{1}$.
Recall that  $a_1a_2 > 4 $.
Thus we obtain $\sqrt{a_1^2a_2^2-4a_1a_2}> a_1a_2-3$,
and hence
\begin{equation}\label{eq.alp2}
	\ga_0 =\al
	=\frac{a_1a_2+\sqrt{a_1^2a_2^2-4a_1a_2}}{2a_2}
	>\frac{2 a_1a_2-3}{2a_2}
	=a_1 - \frac{3}{2a_2}.
\end{equation}
Assume that $a_1, a_2 \geq 2 $. Then $a_1 - 3/2a_2 > a_1-1>0$.
By the definition of $\lambda$,
either $p_0 \leq  p_1 < (a_1-1)p_0$ or $p_1 < p_0 \leq (a_2-1)p_1$ holds.
In both cases, we deduce that $\ga_0 p_{0} - p_{1} \geq 0$.
Assume that $a_1=1$ \resp{$a_2=1$}. Then $a_1 - 3/2a_2 > 1/2$ \resp{$a_1 - 3/2a_2  > a_1-2$}.
By the definition of $\lambda$, we have $2p_1 \leq p_0 \leq (a_2-2)p_1$
\resp{$2p_0 \leq p_1 \leq (a_1-2)p_0$}.
Hence
we deduce that  $\ga_0 p_{0} - p_{1}  \geq 0$.
Thus we get \eqref{eq.exsinplyzil4}.
Therefore, it remains to show that \eqref{eq.exsinplyzil1}, \eqref{eq.exsinplyzil2}, and \eqref{eq.exsinplyzil3}.

Now, since $\{ \vz \in \ipl \mid \vz^\ast \text{ is extremal}\}$ is a subcrystal of $\ipl$, it follows that $\vx \in \{ \vz \in \ipl \mid \vz^\ast \text{ is extremal}\}$.
Also, by Proposition \ref{prop.element}, we have
\begin{equation}
	\label{eq.1vyastep1} x_j=\ep_{i_j}(\e_{i_{j-1}}^{x_{j-1}}\cdots \e_{i_2}^{x_{2}} \e_{i_1}^{x_{1}} \hx^\ast ) \text{ for } j \geq 1.
\end{equation}

\begin{proposition}[proof of \eqref{eq.exsinplyzil1}]\label{prop.pkxp}
	Let $\vx= \hx \otimes t_\lambda \otimes z_{-\infty} \in \ipl$,
	and write $\hx=(\ldots,  x_2,  x_1)$.
	If $\vx^\ast$ is extremal,
	then  $p_{k} - x_{k} \geq 0$, and
	\begin{equation}
		S_{w_k}\vx^\ast =\e_{i_k}^{x_k}\cdots \e_{i_2}^{x_2} \e_{i_1}^{x_1}\hx^\ast \otimes t_\mu \otimes \e_{i_k}^{p_k-x_k}  \cdots \e_{i_2}^{p_2-x_2} \e_{i_1}^{p_1-x_1}z_{-\infty}
	\end{equation}
	for $k \geq 1$, where $\mu := -\la-\wt(\hx)$.
\end{proposition}

\begin{proof}
	We proceed by induction on $k$. Assume that  $k=1$.
	Since $\vx^\ast  =\hx^\ast \otimes t_\mu \otimes z_{-\infty}$,
	and $\pair{\wt(\vx^\ast)}{\alc_{i_1}}=\pair{-\la}{\alc_{i_1}}=-p_1$,
	we see by \eqref{eq.tnsrep} and \eqref{eq.1vyastep1} that
	\begin{equation}\label{eq.maxxp}
		\ep_1(\vx^\ast)
		=\max\{ x_1, 0 -(-p_1) 	\}
		=\max\{x_1,  p_1 \}.
	\end{equation}
	Because $\vx^\ast$ is extremal, the inequality $\pair{\wt(\vx^\ast)}{\alc_{i_1}}=-p_1 \leq 0$ implies that $\ep_{1}(\vx^\ast)=p_1$.
	By \eqref{eq.maxxp}, we obtain $p_1=\max\{x_1, p_1\}$,
	and hence $x_1 \leq  p_1 $.
	Also we see by \eqref{eq.eic} that
	$S_{w_1}\vx^\ast
	= \e_{i_1}^{p_{1}}\vx^\ast
	% =  \e_{i_1}^{p_{1}} \hx^\ast \otimes t_\mu \otimes z_{-\infty}
	=  \e_{i_1}^{x_{1}} \hx^\ast \otimes t_\mu \otimes \e_{i_1}^{p_{1}-x_{1}} z_{-\infty}$.

	Let $k \geq 2$.
	By the induction hypothesis, we have
	\begin{equation}
		S_{w_{k-1}}\vx^\ast=
			\e_{i_{k-1}}^{x_{k-1}}\cdots \e_{i_2}^{x_2} \e_{i_1}^{x_1}\hx^\ast \otimes t_\mu \otimes \e_{i_{k-1}}^{p_{k-1}-x_{k-1}}  \cdots \e_{i_2}^{p_2-x_2} \e_{i_1}^{p_1-x_1}z_{-\infty}.
	\end{equation}
	Hence we see  by  \eqref{eq.1vyastep1} that,
	\begin{align}
		&\ep_{i_k}(S_{w_{k-1}}\vx^\ast)\\
	 	={}& \max\{
			\ep_{i_k}(\e_{i_{k-1}}^{x_{k-1}}\cdots \e_{i_1}^{x_{1}} \hx^\ast),  \underbrace{
				\ph_{i_k}(\e_{i_{k-1}}^{p_{k-1}-x_{k-1}}\cdots \e_{i_1}^{p_1-x_1}z_{-\infty}) -\pair{\wt(S_{w_{k-1}}\vx^\ast)}{\alc_{i_k}}
			}_{=: \;  m_k}
	   \}\\
		={}& \max\{x_k,  m_k \}.
	\end{align}
	By \eqref{eq.pairsxalc}, we have $\pair{\wt(S_{w_{k-1}} \vx^\ast )}{\alc_{i_k}} =-p_k$.
	Because $\vx^\ast$ is extremal, the inequality
	$\pair{\wt(S_{w_{k-1}} \vx^\ast )}{\alc_{i_k}} =-p_k\leq 0$ implies that $\ep_{i_k}(S_{w_{k-1}}\vx^\ast)=p_k$.
	Hence we obtain
	$p_k=\max\{x_k, m_k \}$,
	which implies
	 $x_k \leq  p_k$.
	Therefore
	we see by \eqref{eq.eic} and  \eqref{eq.sw} that
	% $S_{w_k}\vx^\ast = \e_{i_k}^{p_{k}} S_{w_{k-1}}\vx^\ast  $
	\begin{align}
		S_{w_k}\vx^\ast
		&=   \e_{i_k}^{p_{k}} S_{w_{k-1}}\vx^\ast \\
		&=  \e_{i_k}^{p_{k}} \e_{i_{k-1}}^{x_{k-1}}\cdots \e_{i_2}^{x_2} \e_{i_1}^{x_1}\hx^\ast \otimes t_\mu \otimes \e_{i_{k-1}}^{p_{k-1}-x_{k-1}}  \cdots \e_{i_2}^{p_2-x_2} \e_{i_1}^{p_1-x_1}z_{-\infty}\\
		&= \e_{i_k}^{x_k}\cdots \e_{i_2}^{x_2} \e_{i_1}^{x_1}\hx^\ast \otimes t_\mu \otimes \e_{i_k}^{p_k-x_k}  \cdots \e_{i_2}^{p_2-x_2} \e_{i_1}^{p_1-x_1}z_{-\infty}.
	\end{align}
	Thus we have proved the proposition.
\end{proof}

\begin{proposition}[proof of \eqref{eq.exsinplyzil3}]
	Let $\vx= \hx \otimes t_\lambda \otimes z_{-\infty} \in \ipl$,
	and write $\hx=(\ldots,  x_2,  x_1)$.
	% Assume that $p_{k} - x_{k} \geq 0$  for $k \geq 1$.
	If $\vx^\ast$ is extremal,
	% Let $\vx= \hx \otimes t_\lambda \otimes z_{-\infty} \in \{ \vx \in \ipl \mid \vx^\ast \text{ is extremal}\} $,
	% and set $\hx=(\ldots,  x_2,  x_1)$.
	then $\ga_{k+1} p_{k+1} -p_{k} +x_{k} - \ga_{k+1} x_{k+1} \geq 0$
	for $k \geq 1$.
\end{proposition}

\begin{proof}
	By Corollary \ref{cor.equiv}, it suffices to show that
	$\ph_{i_{j}}(\e_{i_{j-1}}^{p_{j-1}-x_{j-1}}  \cdots \e_{i_2}^{p_2-x_2} \e_{i_1}^{p_1-x_1}z_{-\infty})=0$
	for all $j \geq 1$.
 	Let $j \geq 1$.
	Since $\vx^\ast$ is extremal,
	and since $\pair{\wt(S_{w_{j-1}}\vx^\ast)}{\alc_{i_{j}}}=-p_{j} \leq 0$ by \eqref{eq.pairsxalc},
	we see that $\ph_{i_{j}}(S_{w_{j-1}}\vx^\ast) =0$.
	We know from Proposition \ref{prop.pkxp} that
	\begin{equation}
		S_{w_{j-1}}\vx^\ast =\e_{i_{j-1}}^{x_{j-1}}\cdots \e_{i_2}^{x_2} \e_{i_1}^{x_1}\hx^\ast \otimes t_\mu \otimes \e_{i_{j-1}}^{p_{j-1}-x_{j-1}}  \cdots \e_{i_2}^{p_2-x_2} \e_{i_1}^{p_1-x_1}z_{-\infty}.
	\end{equation}
	We see by \eqref{eq.tnsrph} that
	\begin{align}
		0&=\ph_{i_{j}}(S_{w_{j-1}}\vx^\ast)\\
		&= \max \{
			\ep_{i_{j}}(e_{i_{j-1}}^{x_{j-1}} \cdots  \e_{i_1}^{x_1}\hx^\ast)
			+\pair{\wt(S_{w_{j-1}}\vx^\ast)}{\alc_{i_{j}}},
			\ph_{i_{{j-1}}}(\e_{i_{j-1}}^{p_{j-1}-x_{j-1}}  \cdots  \e_{i_1}^{p_1-x_1}z_{-\infty})
			 \}.
	\end{align}
	% Since $x_j-p_j<0$ by Proposition \ref{prop.pkxp},
	Hence we obtain $ 0\geq \ph_{i_{j}}(\e_{i_{j-1}}^{p_{j-1}-x_{j-1}}  \cdots \e_{i_2}^{p_2-x_2} \e_{i_1}^{p_1-x_1}z_{-\infty})$.
	Because $\ph_i(\hz)\geq 0$ for all $i \in I$ and $\hz \in \Zm$,
	we conclude that $ 0 =\ph_{i_{j}}(\e_{i_{j-1}}^{p_{j-1}-x_{j-1}}  \cdots \e_{i_2}^{p_2-x_2} \e_{i_1}^{p_1-x_1}z_{-\infty})$.
	Thus we have proved the proposition.
\end{proof}

\begin{proposition}[proof of \eqref{eq.exsinplyzil2}]
	Let $\vx= \hx \otimes t_\lambda \otimes z_{-\infty} \in \ipl$,
	and write $\hx=(\ldots,  x_2,  x_1)$.
	If $\vx^\ast$ is extremal,
	then $\ga_{k}  x_{k}-x_{k+1} \geq 0$
	for $k \geq 1$.
\end{proposition}

\begin{proof}
	By Corollary \ref{cor.equiv}, it suffices to show that
	$	\ep_{i_j}(\f_{i_{j+1}}^{p_{j+1}} \f_{i_{j+2}}^{p_{j+2}} \cdots \f_{i_0}^{p_0} \hx^\ast) =0$ for all $j \leq 0$.
 	Let $j \leq 0$.
	Since $\vx^\ast$ is extremal,
	and since $\pair{\wt(S_{w_j}\vx^\ast)}{\alc_{i_{j}}}=p_{j} \geq 0$ by \eqref{eq.pairsxalc},
	we see that $\ep_{i_{j}}(S_{w_j}\vx^\ast) = 0 $.
	We see by \eqref{eq.sw} that
	\begin{equation}
		S_{w_j}\vx^\ast
		=\f_{i_{j+1}}^{p_{j+1}}  \cdots \f_{i_{-1}}^{p_{-1}} \f_{i_0}^{p_0} S_{w_0}\vx^\ast
		=\f_{i_{j+1}}^{p_{j+1}}  \cdots \f_{i_{-1}}^{p_{-1}} \f_{i_0}^{p_0} (\hx^\ast \otimes t_\mu \otimes z_{-\infty}).
	\end{equation}
 	Since $S_{w_j}\vx^\ast \neq \0$,
	and  since $\f_i z_{-\infty} =\0$ for all $i \in I$, we see that
	\begin{equation}
		S_{w_j}\vx^\ast =\f_{i_{j+1}}^{p_{j+1}}  \cdots \f_{i_{-1}}^{p_{-1}} \f_{i_0}^{p_0} \hx^\ast \otimes t_\mu \otimes z_{-\infty}.
	\end{equation}
	It follows from \eqref{eq.tnsrep} that
	$0=\ep_{i_j}(S_{w_j} \vx^\ast)
	 =  \max \{
		 \ep_{i_j}(\f_{i_{j+1}}^{p_{j+1}}  \cdots \f_{i_{-1}}^{p_{-1}}  \f_{i_0}^{p_0} \hx^\ast),
		 0 -p_{j}
			\}$.
	Since $-p_{j} <0$, and since $\ep_i(\hz)\geq 0$ for all $i \in I$ and $\hz \in \Zp$, we obtain
	\begin{equation}
		\ep_{i_k}(\f_{i_{j+1}}^{p_{j+1}} \f_{i_{j+2}}^{p_{j+2}} \cdots \f_{i_{-1}}^{p_{-1}} \hx^\ast) =0.
	\end{equation}
	Thus we have proved the proposition.
\end{proof}

\appendix
\section*{Appendices.}
\section{Action of $\bars_k$ on $\Xil$.}\label{apd.1}
In this appendix, we compute $\bars_k(\phi)$, $k \in\Z$, for $\phi =c+\sum_{l\in\Z} \phi_l \ze_l \in \Xil $;
recall that $\bars_k(\phi)=\phi$ for $k \in\Z $ such that  $\phi_k=0$.
\begin{align}
	&	\bars_{0}(\ga_0 p_0 +\ga_0 \ze_{0}- \ze_{1}  )
		=\al(\ga_1 \ze_{1}-\ze_2).\\
	&	\bars_{1}(\ga_0 p_0 +\ga_0 \ze_{0}- \ze_{1}  )
		= \frac{1}{\be}(\ga_{-1} p_{-1} -p_0 +\ga_{-1} \ze_{-1} -  \ze_0  ).\\
	&\bars_{0}(\ga_1 p_1 +\ze_{0}-\ga_1 \ze_{1}  )
	= \frac{1}{\al}(\ga_2 p_2 -p_1 +\ze_1-\ga_2 \ze_2 ).\\
	&\bars_{1}(\ga_1 p_1 +\ze_{0}-\ga_1 \ze_{1}   )
	= \be(\ze_{-1}-\ga_0 \ze_0 ).
\end{align}
For $k \geq 1$,
\begin{align}
		&\bars_k(p_k-\ze_k)=
		\begin{dcases*}
			(\ze_{-1}-\ga_0 \ze_0)+  \frac{1}{\be}(-\ze_0) & if  $k =1 $,  \\
			\frac{1}{\al}(p_1-\ze_1)+(\ga_1 p_1+\ze_0-\ga_1 \ze_1) & if $k =2 $,  \\
  		\frac{1}{\ga_k}( p_{k-1} -\ze_{k-1}) +(\ga_{k-1} p_{k-1} -p_{k-2} +\ze_{k-2} - \ga_{k-1} \ze_{k-1})  & if $k \geq 3 $.
		\end{dcases*}\\
		&\bars_k(\ga_{k} \ze_{k}-\ze_{k+1})= \ga_k(\ga_{k+1}\ze_{k+1}-\ze_{k+2}).\\
		&		\bars_{k+1}(\ga_{k} \ze_{k}-\ze_{k+1} )
				= \begin{dcases*}
					\frac{1}{\al}(\ga_0 p_0 +\ga_0 \ze_0 -\ze_1) & if  $k =1 $,  \\
					\frac{1}{\ga_{k-1}}(\ga_{k-1} p_{k-1} - \ze_{k}) & if $k \geq 2 $.
				\end{dcases*}\\
		&		\bars_{k}(\gamma_{k+1} p_{k+1} -p_{k} +\ze_{k} - \gamma_{k+1} \ze_{k+1}  )
				=\frac{1}{\ga_{k+2}}(\gamma_{k+2} p_{k+2} -p_{k+1} +\ze_{k+1} - \gamma_{k+2} \ze_{k+2} ).\\
		&		\bars_{k+1}(\gamma_{k+1} p_{k+1} -p_{k} +\ze_{k} - \gamma_{k+1} \ze_{k+1}  )
				= \begin{dcases*}
					\al(\ga_1 p_1 + \ze_0 -\ga_1 \ze_1) & if  $k =1 $,  \\
					\ga_{k+1}(\ga_{k} p_{k} -p_{k-1}+ \ze_{k+1}-\ga_{k} \ze_{k}) & if $k \geq 2 $.
				\end{dcases*}
\end{align}
For $k\leq 0$,
\begin{align}
		&\bars_k(p_k+\ze_k)=
		\begin{dcases*}
			\frac{1}{\al}( -\ze_1 )+(\ga_1 \ze_1-\ze_2) & if  $k =0 $,  \\
			\frac{1}{\be}(p_0+\ze_0) + (\ga_0 p_0 + \ga_0 \ze_0 - \ze_1) & if $k =-1 $,  \\
  		\frac{1}{\ga_k}( p_{k+1} +\ze_{k+1}) +(\ga_{k+1} p_{k+1} -p_{k+2} + \ga_{k+1} \ze_{k+1}-\ze_{k+2})  & if $k \leq -2 $.
		\end{dcases*}\\
		&\bars_{k-1}( \ze_{k-1} - \gamma_{k} \ze_{k}  )
		= \begin{dcases*}
			\frac{1}{\be}(\ga_1 p_1 + \ze_0 -\ga_1 \ze_1) & if  $k =0 $,  \\
			\frac{1}{\ga_{k+1}}(\ze_{k}-\ga_{k+1}\ze_{k+1}) & if $k \leq -1 $.
		\end{dcases*}\\
	&\bars_{k}( \ze_{k-1} - \gamma_{k} \ze_{k} )
	=\ga_{k}(\ze_{k-2} - \gamma_{k-1} \ze_{k-1} ).\\
	&	\bars_{k-1}(\gamma_{k-1} p_{k-1} + \gamma_{k-1} \ze_{k-1} -p_{k} +\ze_{k} )
		= \begin{dcases*}
			\be(\ga_0 p_0 + \ga_0 \ze_0 - \ze_1) & if  $k = 0 $,  \\
			\ga_{k-1}(\ga_{k} p_{k} -p_{k+1}+\ga_{k} \ze_{k}- \ze_{k+1}) & if $k \leq -1 $.
		\end{dcases*}\\
	&\bars_{k}(\gamma_{k-1} p_{k-1} + \gamma_{k-1} \ze_{k-1} -p_{k} +\ze_{k}  )
	=\frac{1}{\ga_{k-2}}(\gamma_{k-2} p_{k-2} -p_{k-1} + \gamma_{k-2} \ze_{k-2}  -\ze_{k-1} ).
\end{align}

\section{Proof of Theorem \ref{thm.weight} in the case that $a_{1} =1$ or $a_{2}=1$.}\label{apd.2}
We give a proof only for the case that $a_2=1$ (i.e., part (3));
the proof for the case that $a_1=1$ (i.e., part (2)) is similar.
For $\mu=k\Lambda_1-l\Lambda_2 \in P$,
we define the sequence $ \{ p^{\mu}_m \}_{m\in \Z}$ of integers
by the following recursive formulas: for $m\geq 0$,
\begin{equation}\label{eq.pm.apd}
	p^{\mu}_0:=l, \quad
	p^{\mu}_1:=k, \quad
	p^{\mu}_{m+2}:=
	\begin{cases*}
		a_2 p^{\mu}_{m+1}-p^{\mu}_m & if $m$ is even, \\
		a_1 p^{\mu}_{m+1}-p^{\mu}_m & if $m$ is odd; \\
	\end{cases*}
\end{equation}
for  $m<0$,
	\begin{equation}\label{eq.pm2.apd}
		p^{\mu}_{m}=
		\begin{cases*}
			a_2 p^{\mu}_{m+1}-p^{\mu}_{m+2} & if $m$ is even, \\
			a_1 p^{\mu}_{m+1}-p^{\mu}_{m+2} & if $m$ is odd; \\
		\end{cases*}
\end{equation}
note that for $m \in \Z$,
\begin{equation}\label{eq.xm.apd}
	w_m\mu=
	\begin{cases*}
		p^{\mu}_{m+1}\Lambda_1-p^{\mu}_m\Lambda_2 &
		if  $m$   is even,\\
		-p^{\mu}_m\Lambda_1+p^{\mu}_{m+1}\Lambda_2 &
		if  $m$   is odd.
	\end{cases*}
\end{equation}

\begin{lemma}\label{lem.pm.apd}
	Assume that  $a_1\geq 5$ and $a_2=1$. Let $\mu \in P$.
	\begin{enumerate}[\upshape(1)]
		\item If there exists $n \in \Z$ such that $0<p^{\mu}_{2n}\leq p^{\mu}_{2n+2}$,
		then $0<p^{\mu}_{2m}\leq p^{\mu}_{2m+2}$ for all $m\geq n$.
		\item If there exists $n \in \Z$ such that $0<p^{\mu}_{2n}\leq p^{\mu}_{2n-2}$,
		then $0<p^{\mu}_{2m}\leq p^{\mu}_{2m-2}$ for all $m\leq n$.
	\end{enumerate}
\end{lemma}

\begin{proof}
	We give a proof only for part (1); the proof for  part (2) is similar.
	We proceed by induction on $m$.
	If $m=n$, then the assertion is trivial.
	Assume that $m > n$.
	By \eqref{eq.pm.apd} and \eqref{eq.pm2.apd},
	we have $p^{\mu}_{2m+2} - p^{\mu}_{2m}
	= (a_1-3 )(p^{\mu}_{2m}-p^{\mu}_{2m-2})+(a_1-4)p^{\mu}_{2m-2}$.
	Since $p^{\mu}_{2m}-p^{\mu}_{2m-2}\geq 0$ and $p^{\mu}_{2m-2}>0$
	by the induction hypothesis,
	we obtain $p^{\mu}_{2m+2} - p^{\mu}_{2m}> 0$.
\end{proof}

\begin{proof}[Proof of  Theorem \ref{thm.weight} (3).]
	Assume that $O \in \mathbb{O}$ satisfies  condition \eqref{eq.A}.
	We can take $\mu = k\Lambda_1-l\Lambda_2 \in O$ such that $k, l >0$.
	Then we see by the assumption and
	\eqref{eq.xm.apd} that $p^{\mu}_m >0$ for all $m \in \Z$.
	Hence it follows from Lemma \ref{lem.pm.apd} that
	there exists $n\in\Z$ such that
	\begin{equation}\label{eq.p2m.apd}
		\cdots\geq p^{\mu}_{2n-4} \geq p^{\mu}_{2n-2} \geq  p^{\mu}_{2n} \leq p^{\mu}_{2n+2}  \leq p^{\mu}_{2n+4} \leq \cdots.
	\end{equation}
	By \eqref{eq.pm.apd} and \eqref{eq.pm2.apd},
	we have $p^{\mu}_{2n-2} -p^{\mu}_{2n}=(a_1-2)p^{\mu}_{2n}-p^{\mu}_{2n+1}$ and
	$p^{\mu}_{2n+2} - p^{\mu}_{2n} =p^{\mu}_{2n+1} -2p^{\mu}_{2n}$.
	Hence we see by \eqref{eq.p2m.apd} that  $2p^{\mu}_{2n} \leq  p^{\mu}_{2n+1} \leq (a_1-2)p^{\mu}_{2n}$.
	Then, $\lambda:= w_{2n}\mu =p^{\mu}_{2n+1} \Lambda_1 - p^{\mu}_{2n}\Lambda_2 \in W\mu =O$ satisfies the desired condition.

	Let $\lambda= k_1\Lambda_1-k_2\Lambda_2$ for some
	$k_1, k_2 \in \Z_{>0}$ such that $2k_2 \leq  k_1\leq (a_1-2)k_2$;
	we show that $O:=W\lambda$ satisfies  condition \eqref{eq.A}.
	By \eqref{eq.xm.apd},
	it suffices to show that $p^{\lambda}_m > 0$ for all $m \in \Z$.
	By \eqref{eq.pm.apd}, \eqref{eq.pm2.apd}, and the assumption that $2k_2 \leq  k_1\leq (a_1-2)k_2$,
	we obtain  $p^{\lambda}_2-p^{\lambda}_0= p^{\lambda}_1-2p^{\lambda}_0=k_1-2k_2 \geq 0$ and
	$p^{\lambda}_{-2} - p^{\lambda}_0 = (a_1-2) p^{\lambda}_{0}-p^{\lambda}_1=
	(a_1-2)k_2-k_1 \geq 0$.
	Hence we see by Lemma \ref{lem.pm.apd} that $p^{\lambda}_{2m}> 0$ for all $m \in \Z$.
	Note that $p^{\lambda}_{2m-1} = p^{\lambda}_{2m+2} + p^{\lambda}_{2m}$
	by \eqref{eq.pm.apd} and \eqref{eq.pm2.apd}.
	% we have $p^{\lambda}_{2m+2} =p^{\lambda}_{2m-1}-p^{\lambda}_{2m}$,
	% we see that $p^{\lambda}_{2m-1} > p^{\lambda}_{2m}$.
	Since $p^{\lambda}_{2m}, p^{\lambda}_{2m+2}>0 $  as seen above,
 	we get $p^{\lambda}_{2m-1} = p^{\lambda}_{2m+2} + p^{\lambda}_{2m}>0$ for all $m \in \Z$.

	Thus we have proved part (3) of Theorem \ref{thm.weight}.
\end{proof}

\section*{Acknowledgment.}
The author would like to thank Daisuke Sagaki, who is his supervisor, for his kind support and advice.

\end{document}